\documentclass[12pt]{amsart}
\usepackage{amsmath}
\usepackage{amssymb}
\usepackage[all]{xy}
\usepackage{longtable}
\allowdisplaybreaks
\newtheorem{theorem}[equation]{Theorem}
\newtheorem{lemma}[equation]{Lemma}
\newtheorem{proposition}[equation]{Proposition}
\newtheorem{corollary}[equation]{Corollary}

\newtheorem{case}{Case}

\theoremstyle{definition}

\theoremstyle{remark}

\numberwithin{equation}{section}

\DeclareMathAlphabet{\matheur}{U}{eur}{m}{n}

\DeclareMathOperator{\Li}{Li}

\DeclareMathOperator{\Sym}{Sym}
\DeclareMathOperator{\AGM}{AGM}

\DeclareMathOperator{\F}{F}

\newmuskip\pFqskip
\mathchardef\pFcomma=\mathcode`, 

\newcommand*\pFq[5]{%
  \begingroup
  \begingroup\lccode`~=`,
    \lowercase{\endgroup\def~}{\pFcomma\mkern\pFqskip}%
  \mathcode`,=\string"8000
  {}_{#1}F_{#2}\biggl(\genfrac..{0pt}{}{#3}{#4};#5\biggr)%
  \endgroup
}
\renewcommand{\Im}{\operatorname{Im}}
\renewcommand{\Re}{\operatorname{Re}}
\begin{document}

\title[The elliptic trilogarithm and Mahler measures of $K3$ surfaces]{The elliptic trilogarithm and Mahler measures of $K3$ surfaces}

\author{Detchat Samart}
\address{Department of Mathematics, Texas A{\&}M University, College
Station,
TX 77843, USA} \email{detchats@math.tamu.edu}

\thanks{The author's research was partially supported by NSF Grant DMS-1200577}


\date{\today}

\begin{abstract}
The aim of this paper is to derive explicitly a connection between the Zagier elliptic trilogarithm and Mahler measures of certain families of three-variable polynomials defining $K3$ surfaces. In addition, we prove some linear relations satisfied by the elliptic trilogarithm evaluated at torsion points on elliptic curves. The latter result can be viewed as a higher dimensional analogue of \textit{exotic relations} of the elliptic dilogarithm.
\end{abstract}


\maketitle
\section{Introduction}\label{Sec:Intro}
For $m\in\mathbb{N}$, the classical $m^\text{th}$ polylogarithm function is defined by 
\begin{equation*}
\Li_m(z):=\sum_{n=1}^{\infty}\frac{z^n}{n^m}, \qquad |z|<1.
\end{equation*}
One can obtain a multivalued function on $\mathbb{C}-\{0,1\}$ from $\Li_m(z)$ by extending it analytically. There are many versions of higher polylogarithms in the literature, including the following single-valued function defined in \cite[\S 2]{ZG}
\begin{equation*}
\mathcal{L}_m(z)=\mathfrak{R}_m\left(\sum_{k=0}^{m-1}\frac{2^k B_k}{k!}\log^k |z| \Li_{m-k}(z)\right), \qquad |z|\leq 1,
\end{equation*}
where 
\begin{equation*}
\mathfrak{R}_m= \begin{cases}
				 \Re        &\text{if } m \text{ is odd},\\
                 \Im         &\text{if } m \text{ is even},               
              \end{cases}
\end{equation*}
and $B_k$ is the $k^\text{th}$ Bernoulli number $\left(B_0=1, B_1=-1/2, B_2=1/6, B_3=0, \ldots\right)$. It can be extended to a continuous function on $\mathbb{P}^1(\mathbb{C})$ by the functional equation
\begin{equation*}
\mathcal{L}_m\left(z^{-1}\right)=(-1)^{m-1}\mathcal{L}_m(z).
\end{equation*}
For $m=2$, $\mathcal{L}_m(z)$ becomes the \textit{Bloch-Wigner dilogarithm}
\begin{equation*}
D(z)=\Im(\Li_2(z)+\log |z|\log(1-z)).
\end{equation*}
The dilogarithm function and the function $D(z)$ have been studied extensively and are known to satisfy several interesting properties. They are also found to have fruitful applications in algebraic $K$-theory and other related areas (see, for example, \cite{Zagier}). 

Another version of higher polylogarithm functions constructed by Ramakrishnan \cite{Ramakrishnan} and formulated in terms of the polylogarithms by Zagier \cite{ZagierI} is
\begin{equation*}
D_m(z)= \mathfrak{R}_m\left(\sum_{k=0}^{m-1}\frac{(-1)^k}{k!}\log^k|z|\Li_{m-k}(z)-\frac{(-1)^m}{2m!}\log^m|z|\right),
\end{equation*}
where $|z|\leq 1$.

Now let us consider an ``averaged'' version of the function $D(z)$, which will be described below. Let $E$ be an elliptic curve defined over $\mathbb{C}.$ Then there exist $\tau\in\mathcal{H}:=\{\tau\in\mathbb{C}\mid \Im(\tau)>0\}$ and  isomorphisms
\begin{equation}\label{E:isom}
\begin{aligned}
\qquad E(\mathbb{C})\qquad& \tilde{\longrightarrow} & \mathbb{C}/\Lambda\qquad & \tilde{\longrightarrow} & \mathbb{C}^{\times}/q^{\mathbb{Z}}\\
\left(\wp_{\Lambda}(u),\wp_{\Lambda}'(u)\right)\qquad & \longmapsto &  u \pmod \Lambda \qquad & \longmapsto & e^{2\pi i u},
\end{aligned}
\end{equation}
where $\Lambda=\mathbb{Z}+\mathbb{Z}\tau$, and $\wp_{\Lambda}$ denotes the Weierstrass $\wp$-function. Using the transformations above, Bloch \cite{Bloch} defined the \textit{elliptic dilogarithm} $D^{E}:E(\mathbb{C})\rightarrow \mathbb{R}$ by
\begin{equation*}
D^E(P):=D^E(x)=\sum_{n=-\infty}^{\infty}D(q^n x),
\end{equation*}
where $q=e^{2\pi i\tau}$ and $x=e^{2\pi i u}$ is the image of $P$ in  $\mathbb{C}^{\times}/q^{\mathbb{Z}}.$ (Note that the series above converges absolutely with exponential rapidity and is invariant under $x\mapsto qx$ \cite[\S 2]{ZagierI}.) 

Recall from \cite{Weil} that for $a,b\in\mathbb{N}$ the series
\begin{equation*}
K_{a,b}(\tau;u)=\sideset{}{'}\sum_{m,n\in\mathbb{Z}}\frac{e^{2\pi i(n\xi-m\eta)}}{(m\tau+n)^a(m\bar{\tau}+n)^b},
\end{equation*}
where $u=\xi\tau+\eta$ and $\xi,\eta\in\mathbb{R}/\mathbb{Z}$, is called the \textit{Eisenstein-Kronecker series.} (Here and throughout $\displaystyle\sideset{}{'}\sum_{m,n}$ means that the summation does not include $(m,n)=(0,0)$.) 
Then Bloch introduced the regulator function
\begin{equation*}
R^E(e^{2\pi i u})=\frac{\Im(\tau)^2}{\pi}K_{2,1}(\tau;u).
\end{equation*} 
It can be shown that $\Re(R^E)=D^E$, and that $-\Im(R^E)$ is the real-valued function given by 
\begin{equation*}
J^E(x)=\sum_{n=0}^{\infty}J(q^n x)-\sum_{n=1}^{\infty}J(q^n x^{-1})+\frac{1}{3}\log^2|q|B_3\left(\frac{\log|x|}{\log|q|}\right),
\end{equation*}
where $J(x)=\log|x|\log|1-x|$, and $B_n(X)$ denotes the $n^\text{th}$ Bernoulli polynomial. Similar to $D^E$, the function $J^E$ is well-defined and invariant under $x\mapsto qx.$ We can also extend $R^E,D^E,$ and $J^E$ by linearity to the group of divisors on $E(\mathbb{C})$. \\
\indent Recall that for any nonzero Laurent polynomial $P\in \mathbb{C}[X_1^{\pm 1},\ldots,X_n^{\pm 1}]$, the Mahler measure of $P$ is defined by 
\begin{equation*}
m(P)=\int_0^1\cdots \int_0^1 \log |P(e^{2\pi i \theta_1},\ldots,e^{2\pi i \theta_n})|\,d\theta_1\cdots d\theta_n.
\end{equation*}
Let us denote
\begin{align*}
m_2(t)&:=2m(x+x^{-1}+y+y^{-1}+t^{1/2}),\\
m_3(t)&:=3m(x^3+y^3+1-t^{1/3}xy).
\end{align*}
Boyd \cite{Boyd} and Rodriguez Villegas \cite{RV} verified numerically that for many values of $t\in\mathbb{Z}$, $m_2(t)$ and $m_3(t)$ appear to be of the form 
\begin{equation*}
m_j(t)\stackrel{?}\sim_{\mathbb{Q}} L'(E,0), \qquad j=2,3,
\end{equation*} 
where $E$ is the elliptic curve over $\mathbb{Q}$ defined by the zero locus of $x+x^{-1}+y+y^{-1}+t^{1/2}$ or $x^3+y^3+1-t^{1/3}xy$, depending on $j$, and $A\sim_{\mathbb{Q}} B$ means $B=cA$ for some $c\in\mathbb{Q}.$ (Note that by the functional equation of $L(E,s)$, one has $|L'(E,0)|=(2\pi)^{-2}N_{E}L(E,2)$, where $N_E$ is the conductor of $E$.) However, most of these conjectured formulas have remained unproved. Rodriguez Villegas \cite{RV} proved some of these formulas when the corresponding elliptic curves have complex multiplication by showing first that $m_2(t)$ and $m_3(t)$ can be expressed in terms of Eisenstein-Kronecker series. Afterward, Lal\'{i}n and Rogers\cite{LR} verified that for every $t\in\mathbb{C}$, if $E$ is the elliptic curve defined by $x+x^{-1}+y+y^{-1}+t^{1/2}=0$ and $\tau=\omega_2/\omega_1$, where $\omega_1$ and $\omega_2$ are the real and complex periods of $E$, respectively, then 
\begin{equation*}
m_2\left(t^2\right)=\frac{1}{\pi\Im\left(\frac{1}{2\tau}\right)}J^E(e^{-\pi i/2\tau}).
\end{equation*}  
This result then implies the result of Rodriguez Villegas above immediately. More recently, Guillera and Rogers \cite[Thm.~5]{GR} proved that $m_j(t), t=2,3,$ can be rephrased explicitly in terms of the elliptic dilogarithm evaluated at torsion points. For instance, they showed that 
if $E(k,l)$ is the elliptic curve given by the equation
\begin{equation*}
y^2=4x^3-27(k^4-16k^2+16)l^2 x+27(k^6-24k^4+120k^2+64)l^3,
\end{equation*}
and $u=18,25,64,$ and $256$, then $m_2(u)$ can be expressed as
\begin{equation*}
m_2(u)=\frac{8}{\pi}D^{E(\sqrt{u},l)}(P),\\
\end{equation*}
where $l\in\{1/2,1,2\}$ and $P$ is a $4$-torsion point on the corresponding elliptic curve. They also proved that if $E$ is the elliptic curve $y^2=4x^3-432x+1188$ and $P=(-6,54)$, then the Mahler measure identity
\begin{equation*}
16m_3\left(\frac{(7+\sqrt{5})^3}{4}\right)-8m_3\left(\frac{(7-\sqrt{5})^3}{4}\right)=19m_3(32)
\end{equation*}
is equivalent to the exotic relation
\begin{equation*}
16D^E(P)-11D^E(2P)=0,
\end{equation*}
verified by Bertin in \cite{Bertin}. 

In the present paper, we will investigate the elliptic version of 
\begin{equation*}
\mathcal{L}_3(z)=\Re\left(\Li_3(z)-\log|z|\Li_2(z)+\frac{1}{3}\log^2|z|\Li_1(z)\right),
\end{equation*}
namely \textit{the elliptic trilogarithm}
\begin{equation*}
\mathcal{L}^E_{3,1}(P):=\mathcal{L}_{3,1}^E(x)=\sum_{n=-\infty}^{\infty}\mathcal{L}_3(q^n x),
\end{equation*}
where $E$ is again identified with $\mathbb{C}/(\mathbb{Z}+\mathbb{Z}\tau)$, $u$ is the image of $P$ in $\mathbb{C}/(\mathbb{Z}+\mathbb{Z}\tau)$, $q=e^{2\pi i \tau}$, and $x=e^{2\pi i u}$. This function was first defined by Zagier in \cite[\S 10]{ZG}. He also introduced a companion of $\mathcal{L}^E_{3,1}$, which is the following single-valued function:
\begin{equation*}
\mathcal{L}^E_{3,2}(P):=\mathcal{L}^E_{3,2}(x)=\sum_{n=0}^{\infty}J_3(q^n x)+\sum_{n=1}^{\infty}J_3(q^n x^{-1})+\frac{\log^2|x|\log^2|qx^{-1}|}{4\log|q|},
\end{equation*}
where $J_3(x)=\log^2|x|\log|1-x|.$ Again, one can extend $\mathcal{L}^E_{3,j},j=1,2,$ to all divisors on $E(\mathbb{C})$ by linearity. Zagier claimed that $\mathcal{L}^E_{3,1}$ and $\mathcal{L}^E_{3,2}$ are linear combinations of the Eisenstein-Kronecker series $K_{1,3}$ and $K_{2,2}$, and we shall derive this result explicitly in Section~\ref{Sec:elltrilog}. Furthermore, we will deduce an integer relation satisfied by $\mathcal{L}^E_{3,1}$ evaluated at the torsion points of order $2,3,$ and $4$ on elliptic curves. 

In Section~\ref{Sec:MM} we draw a connection between the elliptic trilogarithm and Mahler measures of three-variable polynomials by proving the following result, which is the main result of this paper: 
\begin{theorem}\label{T:Main}
Denote
\begin{align*}
n_2(s)&:=2m((x+x^{-1})(y+y^{-1})(z+z^{-1})+s^{1/2}),\\
n_3(s)&:=m\left((x+x^{-1})^2(y+y^{-1})^2(1+z)^3z^{-2}-s\right),\\
n_4(s)&:=4m\left(x^4+y^4+z^4+1+s^{1/4}xyz\right).
\end{align*} 
\begin{enumerate}
\item[(i)]
Let $E_s$ be the family of elliptic curves given by 
$E_s: y^2=(x-1)\left(x^2-\frac{s}{s-64}\right).$
If $s\in\mathbb{R}\backslash[0,64]$, $r:=\sqrt{\frac{s}{s-64}}$, $P:=(-r,0)$, and $Q:=(r,0)$, then
\begin{equation}\label{E:Main}
n_2(s)=\frac{8}{3\pi^2}\left(6\mathcal{L}^{E_s}_{3,1}((P)-(Q))-\mathcal{L}^{E_s}_{3,2}((P)-(Q))\right).
\end{equation}
\item[(ii)]
Let $F_s$ be the family of elliptic curves given by 
$F_s: x^3+y^3+1-rxy=0,$ where $r=\sqrt[3]{\frac{s+\sqrt{s(s-108)}}{2}}$. If $O,P$ and $Q$ are the points on $F_s$ corresponding to $1,1/3$ and $\tau/3$, respectively, via the isomorphism $F_s\cong \mathbb{C}/(\mathbb{Z}+\mathbb{Z}\tau), \tau\in\mathcal{H}$, then for all $s\in [108,\infty)$
\begin{equation}\label{E:Main2}
n_3(s)=\frac{3}{4\pi^2}\left(15\mathcal{L}^{F_s}_{3,1}((Q)-3(P)-6(P+Q))+\mathcal{L}^{F_s}_{3,2}(3(P)+6(P+Q)-7(Q)-2(O))\right).
\end{equation}
\item[(iii)]
Let $G_s$ be be the family of elliptic curves given by 
$G_s: y^2=(x-1)(x-r')(x+r'),$ where $r'=\sqrt{\frac{1+\sqrt{1-\frac{256}{s}}}{2}}$ and let $P$ and $Q$ denote the points on $G_s$ corresponding to $\tau/2$ and $3/4$, respectively, via the isomorphism $G_s\cong \mathbb{C}/(\mathbb{Z}+\mathbb{Z}\tau)$. If $s\geq 256$, then
\begin{equation}\label{E:Main31}
\begin{aligned}
n_4(s)&=\frac{16}{9\pi^2}\big(15\mathcal{L}^{G_s}_{3,1}(2(P)-(2Q)+2(P+2Q))\\
&\qquad +\mathcal{L}^{G_s}_{3,2}(4(Q)-5(P)+2(2Q)+4(P+Q)-5(P+2Q))\big).
\end{aligned}
\end{equation}
On the other hand, if $s<0$, then
\begin{equation}\label{E:Main32}
\begin{aligned}
n_4(s)&=\frac{8}{9\pi^2}\big(30\mathcal{L}^{G_s}_{3,1}(2(2Q)-(P+2Q)+2(P))\\
&\qquad +\mathcal{L}^{G_s}_{3,2}(5(P+2Q)+8(Q)+8(P+Q)-11(2Q)-10(P))\big).
\end{aligned}
\end{equation}
\end{enumerate}
\end{theorem}
The families of elliptic curves $E_s,F_s,$ and $G_s$ are indeed related to the $K3$ surfaces defined by the corresponding polynomials of which Mahler measures are considered via Shioda-Inose structures. More details about this type of relations will also be given in Section~\ref{Sec:MM}.\\
\indent The author has verified in \cite{Samart,SamartII} that for some real values of $s$, $n_j(s)$ $(j=2,3,4)$ are rational linear combinations of $L'(g,0)$ and $L'(\chi,-1)$, where $g$ is a CM newform of weight $3$ and $\chi$ is a Dirichlet character. Also, many conjectural Mahler measure formulas of this type have been found via numerical computations. Therefore, we obtain explicit relations between the elliptic trilogarithm and those $L$-values. We shall also list some conjectural formulas of evaluations of $\mathcal{L}^{E_s}_{3,j}, j=1,2,$ in terms of special values of $L$-functions discovered via our numerical computations in Section~\ref{Sec:Lvalues}.

\section{The functions $\mathcal{L}^E_{3,1}$ and $\mathcal{L}^E_{3,2}$}\label{Sec:elltrilog}
Most components of the results in this section are deduced from Zagier's results in \cite{ZagierI}. Thus let us first recall some notations and facts obtained from that paper. For any $a,b,l,m\in\mathbb{N}$ with $1\leq a,m\leq l$ and $x,q\in\mathbb{C}$ with $|q|<1$, Zagier defined
\begin{align*}
c_{a,m}^{(l)}&=\sum_{h=1}^a (-1)^{h-1}\binom{m-1}{h-1}\binom{l-m}{a-h},\\
D^{*}_m(x)&= \begin{cases}
				 D_m(x)       &\text{if } m \text{ is odd},\\
                 iD_m(x)      &\text{if } m \text{ is even},               
              \end{cases}\\
D_{a,b}(x)&=2\sum_{m=1}^r c_{a,m}^{(r)}D^{*}_m(x)\frac{(-\log|x|)^{r-m}}{(r-m)!}+\frac{(-2\log|x|)^r}{2\cdot r!},\qquad r=a+b-1,\\
D_{a,b}(q;x)&=\sum_{n=0}^{\infty}D_{a,b}(q^n x)+(-1)^{r-1}\sum_{n=1}^{\infty}D_{a,b}(q^n x^{-1})+\frac{(-2\log|q|)^r}{(r+1)!}B_{r+1}\left(\frac{\log|x|}{\log|q|}\right).
\end{align*}
\begin{proposition}[{Zagier \cite[\S 2]{ZagierI}}]\label{P:Zagier}
Unless otherwise stated, let $a,b\in\mathbb{N}$ and $r=a+b-1$.
\begin{enumerate}
\item[(i)] $D_{a,b}$ is a single-valued real-analytic function on $\mathbb{C}\backslash [1,\infty)$ and satisfies the functional equation
\begin{equation*}
D_{a,b}(x^{-1})=(-1)^{r-1}D_{a,b}(x)+\frac{(2\log|x|)^r}{r!}.
\end{equation*}
\item[(ii)] For any $x\in\mathbb{C},m\geq 1,$ and $n\geq 0$, we have the inversion formula
\begin{equation*}
D^{*}_m(x)\frac{(-\log|x|)^n}{n!}=\sum_{\substack{a,b\geq 1\\ a+b=r+1}}c_{m,a}^{(r)}\left(\frac{D_{a,b}(x)}{2^r}-\frac{(-\log|x|)^r}{2\cdot r!}\right), \quad \text{where } r=m+n.
\end{equation*}
\item[(iii)] Let $q=e^{2\pi i\tau}$ and $x=e^{2\pi i u}$, where $\tau\in\mathcal{H}$ and $u=\xi\tau+\eta, \xi,\eta\in\mathbb{R}$. Then
\begin{equation*}
D_{a,b}(q;x)=\frac{(\tau-\bar{\tau})^r}{2\pi i}K_{a,b}(\tau;u).
\end{equation*}
\end{enumerate}
\end{proposition}
It was pointed out in \cite[\S 10]{ZG} that $\mathcal{L}^E_{3,1}$ and $\mathcal{L}^E_{3,2}$ are linear combinations of $K_{1,3}$ and $K_{2,2}$, but this fact does not seem to be shown in the literature. Therefore, we will give an account of it here before applying it to prove other results.

\begin{proposition}
Suppose that $E\cong \mathbb{C}/(\mathbb{Z}+\mathbb{Z}\tau)$ with $\tau\in\mathcal{H}$. Let $q=e^{2\pi i\tau}$ and $x\in\mathbb{C}.$ Then the following identities hold:
\begin{align}
6\mathcal{L}^E_{3,1}(x)-\mathcal{L}^E_{3,2}(x)&=\frac{3}{4}\left(2\Re\left(D_{1,3}(q;x)\right)-D_{2,2}(q;x)\right)-\frac{\log^3|q|}{120}\label{E:EK1},\\
\mathcal{L}^E_{3,1}(x)&=\frac{1}{6}\left(\Re\left(D_{1,3}(q;x)\right)-D_{2,2}(q;x)\right)\label{E:EK2},\\
\mathcal{L}^E_{3,2}(x)&=-\frac{1}{4}\left(2\Re\left(D_{1,3}(q;x)\right)+D_{2,2}(q;x)\right)+\frac{\log^3|q|}{120}\label{E:EK3}.
\end{align}
\end{proposition}
\begin{proof}
It suffices to prove any two equalities of the above, so we will show \eqref{E:EK1} and \eqref{E:EK2} only. First, using the identity (34) in \cite{ZagierII}, one has 
\begin{equation}\label{E:L3}
\mathcal{L}_3(x)=D_3(x)-\frac{\log^2|x|D_1(x)}{6}.
\end{equation}
It then follows from Proposition~\ref{P:Zagier}(ii) that
\begin{equation}\label{E:D3}
D_3(x)=\frac{1}{8}\left(D_{1,3}(x)+D_{3,1}(x)-D_{2,2}(x)\right)+\frac{\log^3|x|}{12}.
\end{equation}
Since $D_1(x)=-\log\left|x^{1/2}-x^{-1/2}\right|$, we can deduce 
\begin{equation*}
\mathcal{L}_3(x)=\frac{1}{8}\left(D_{1,3}(x)+D_{3,1}(x)-D_{2,2}(x)\right)+\frac{J_3(x)}{6}.
\end{equation*}
Next, by simple manipulations and Proposition~\ref{P:Zagier}(i), we have that
\begin{align*}
\mathcal{L}^E_{3,1}(x)&=\sum_{n\in\mathbb{Z}}\mathcal{L}_3(q^n x)\\
&=\frac{1}{8}\sum_{n\in\mathbb{Z}}\left(D_{1,3}(q^n x)+D_{3,1}(q^n x)-D_{2,2}(q^n x)\right)+\frac{1}{6}\sum_{n\in\mathbb{Z}}J_3(q^n x)\\
&=\frac{1}{8}\sum_{n\geq 0}\left(D_{1,3}(q^n x)+D_{3,1}(q^n x)-D_{2,2}(q^n x)\right)+\frac{1}{6}\sum_{n\geq 0}J_3(q^n x)\\
&\quad +\frac{1}{8}\sum_{n\geq 1}\left(D_{1,3}(q^n x^{-1})+D_{3,1}(q^n x^{-1})-D_{2,2}(q^n x^{-1})\right)+\frac{1}{6}\sum_{n\geq 1}J_3(q^n x^{-1})\\
&= \frac{1}{8}\left(D_{1,3}(q;x)+D_{3,1}(q;x)-D_{2,2}(q;x)\right)+\frac{\mathcal{L}^E_{3,2}(x)}{6}-\frac{\log^3|q|}{720}.
\end{align*}
Now we obtain \eqref{E:EK1} by using Proposition~\ref{P:Zagier}(iii) and the fact that $K_{3,1}(\tau;u)=\overline{K_{1,3}(\tau;u)}.$ To prove \eqref{E:EK2}, we again start with the equation \eqref{E:L3}. Applying Proposition~\ref{P:Zagier}(ii) with $m=1$ and $n=2$, we get
\begin{equation*}
D_1(x)\frac{\log^2|x|}{2}=\frac{1}{8}\left(D_{1,3}(x)+D_{3,1}(x)+D_{2,2}(x)\right)+\frac{\log^3|x|}{4}.
\end{equation*}
Therefore, by \eqref{E:D3},
\begin{equation*}
\mathcal{L}_3(x)=\frac{1}{12}\left(D_{1,3}(x)+D_{3,1}(x)-2D_{2,2}(x)\right).
\end{equation*}
Then one can prove \eqref{E:EK2} easily using similar arguments above.
\end{proof}

\begin{corollary}\label{Co:2}
With the same settings in Proposition~\ref{P:Samart}, if $x=e^{2\pi i(\xi\tau+\eta)}$, where $\xi,\eta\in\mathbb{R}/\mathbb{Z}$, then the following identities hold: 
\begin{align}
\mathcal{L}_{3,1}^E(x)&=\frac{2\Im(\tau)^3}{3\pi}\left[\sideset{}{'}\sum_{m,n\in\mathbb{Z}}\frac{e^{2\pi i(n\xi-m\eta)}}{|m\tau+n|^4}-\Re\left(\sideset{}{'}\sum_{m,n\in\mathbb{Z}}e^{2\pi i(n\xi-m\eta)}\frac{(m\tau+n)^2}{|m\tau+n|^6}\right)\right],\label{E:L31}\\
\mathcal{L}_{3,2}^E(x)&=\frac{\Im(\tau)^3}{\pi}\left[\sideset{}{'}\sum_{m,n\in\mathbb{Z}}\frac{e^{2\pi i(n\xi-m\eta)}}{|m\tau+n|^4}+2\Re\left(\sideset{}{'}\sum_{m,n\in\mathbb{Z}}e^{2\pi i(n\xi-m\eta)}\frac{(m\tau+n)^2}{|m\tau+n|^6}\right)\right]+\frac{\log^3|q|}{120} \label{E:L32}.
\end{align}
\end{corollary}
\begin{proof}
Use \eqref{E:EK2}, \eqref{E:EK3} and Proposition~\ref{P:Zagier}(iii).
\end{proof}
As an easy consequence of \eqref{E:L31}, we have the following result:

\begin{theorem}
\begin{enumerate}
\item[(i)]
Let $E$ be an elliptic curve given by the equation
\begin{equation*}
y^2=4(x-e_1)(x-e_2)(x-e_3),
\end{equation*}
where $e_j\in\mathbb{C}$ are pairwise distinct, and denote by $P_j$ and $O$ the point $(e_j,0)$ and the point at infinity, respectively. Then 
\begin{equation}\label{E:linear1}
\mathcal{L}^E_{3,1}(4(P_1)+4(P_2)+4(P_3)+3(O))=0.
\end{equation} 
\item[(ii)] Let $E$ be an elliptic curve isomorphic to $\mathbb{C}/(\mathbb{Z}+\mathbb{Z}\tau).$ If $P,Q,R,S,$ and $O$ are the points on $E$ corresponding to $\tau/3,1/3,\tau/2,3/4$ and $1$, respectively, via the isomorphism above, then
\begin{equation}
\mathcal{L}^E_{3,1}(9(P)+9(Q)+18(P+Q)+4(O))=0, \label{E:linear2}
\end{equation} 
Moreover, if $\tau$ is purely imaginary, then
\begin{equation}
\mathcal{L}^E_{3,1}(8(S)+8(R+S)-(2S))=0. \label{E:linear3}
\end{equation}
\end{enumerate}
\end{theorem}
\begin{proof}
We shall prove (i) first. Denote by $\omega_1$ and $\omega_2$ the real and complex periods of $E$ and let $\Lambda=\mathbb{Z}\omega_1+\mathbb{Z}\omega_2.$ Then it follows from a well-known fact about evaluations of $\wp_\Lambda$ and $\wp_\Lambda'$ at the half-periods of $\Lambda$ that 
\begin{equation*}
\left\{(\wp_\Lambda(u),\wp_\Lambda'(u))\mid u=\frac{\omega_1}{2},\frac{\omega_2}{2},\frac{\omega_1+\omega_2}{2}\right\}=\{P_1,P_2,P_3\}.
\end{equation*}
Let $\tau=\omega_2/\omega_1$. Then, by using \eqref{E:L31}, we find that
\begin{align*}
\mathcal{L}^E_{3,1}((P_1)+(P_2)+(P_3))&=\frac{2\Im(\tau)^3}{3\pi}\Biggl[\sideset{}{'}\sum_{m,n\in\mathbb{Z}}\frac{(-1)^n+(-1)^m+(-1)^{m+n}}{|m\tau+n|^4}\\
&\qquad -\Re\left(\sideset{}{'}\sum_{m,n\in\mathbb{Z}}\left((-1)^n+(-1)^m+(-1)^{m+n}\right)\frac{(m\tau+n)^2}{|m\tau+n|^6}\right)\Biggr]\\&=\frac{2\Im(\tau)^3}{3\pi}\Biggl[\Re\left(\sideset{}{'}\sum_{m,n\in\mathbb{Z}}\frac{(m\tau+n)^2}{|m\tau+n|^6}-4\sideset{}{'}\sum_{\substack{m \text{ even}\\n \text{ even}}}\frac{(m\tau+n)^2}{|m\tau+n|^6}\right)\\
&\qquad -\left(\sideset{}{'}\sum_{m,n\in\mathbb{Z}}\frac{1}{|m\tau+n|^4}-4\sideset{}{'}\sum_{\substack{m \text{ even}\\n \text{ even}}}\frac{1}{|m\tau+n|^4}\right)\Biggr]\\
&=\frac{\Im(\tau)^3}{2\pi}\Biggl[\Re\left(\sideset{}{'}\sum_{m,n\in\mathbb{Z}}\frac{(m\tau+n)^2}{|m\tau+n|^6}\right)-\sideset{}{'}\sum_{m,n\in\mathbb{Z}}\frac{1}{|m\tau+n|^4}\Biggr]\\
&=-\frac{3}{4}\mathcal{L}^E_{3,1}(O),
\end{align*}
so \eqref{E:linear1} follows. Next, suppose that $E\cong \mathbb{C}/(\mathbb{Z}+\mathbb{Z}\tau).$ Using \eqref{E:L31} and the fact that $\mathcal{L}^E_{3,1}$ is a real-valued function, one has that
\begin{align*}
\mathcal{L}^E_{3,1}(8(S)+8(R+S))&=\frac{16y^3}{3\pi}\Biggl[\sideset{}{'}\sum_{m,n\in\mathbb{Z}}\frac{i^m\left(1+(-1)^n\right)}{|m\tau+n|^4}\\
&\qquad -\Re\left(\sideset{}{'}\sum_{m,n\in\mathbb{Z}}\left(i^m\left(1+(-1)^n\right)\right)\frac{(m\tau+n)^2}{|m\tau+n|^6}\right)\Biggr]\\
&=\frac{16y^3}{3\pi}\Biggl[\sideset{}{'}\sum_{m,n\in\mathbb{Z}}\frac{(-1)^m\left(1+(-1)^n\right)}{|2m\tau+n|^4}\\
&\qquad -\sideset{}{'}\sum_{m,n\in\mathbb{Z}}\left((-1)^m\left(1+(-1)^n\right)\right)\frac{n^2-4y^2m^2}{|2m\tau+n|^6}\Biggr]\\
&=\frac{2y^3}{3\pi}\sideset{}{'}\sum_{m,n\in\mathbb{Z}}\left(\frac{(-1)^m}{|m\tau+n|^4}-\frac{(-1)^m(n^2-y^2m^2)}{|m\tau+n|^6}\right)\\
&=\mathcal{L}^E_{3,1}(2S),
\end{align*}
which yields \eqref{E:linear3}. On the other hand, if $\omega=e^{2\pi i/3}$ and $\tau=y i$, where $y\in\mathbb{R},$ then
\begin{align*}
\mathcal{L}^E_{3,1}(9(P)+9(Q)+18(P+Q)&+4(O))=\frac{2y^3}{3\pi}\Biggl[\sideset{}{'}\sum_{m,n\in\mathbb{Z}}\frac{9\omega^n+9\omega^{-m}+18\omega^{n-m}+4}{|m\tau+n|^4}\\
&\qquad -\Re\left(\sideset{}{'}\sum_{m,n\in\mathbb{Z}}\left(9\omega^n+9\omega^{-m}+18\omega^{n-m}+4\right)\frac{(m\tau+n)^2}{|m\tau+n|^6}\right)\Biggr]\\
&=\frac{2y^5}{3\pi}\sideset{}{'}\sum_{m,n\in\mathbb{Z}}\frac{\left(80-27|\chi(n)|-27|\chi(m)|-54|\chi(n-m)|\right)m^2}{(n^2+y^2m^2)^3},
\end{align*} 
where $\chi$ is the quadratic character of conductor $3$. To obtain the last latter equality, we use the identity 
\begin{equation}\label{E:omega}
\omega^n = -\frac{3}{2}|\chi(n)|+1+\chi(n)\frac{\sqrt{3}}{2}i, \qquad n\in\mathbb{Z}.
\end{equation}
Now it can be shown that the last series above vanishes by considering $80-27|\chi(n)|-27|\chi(m)|-54|\chi(n-m)|$ for each $m$ and $n$ modulo $3$.
\end{proof}

\section{Connection with Mahler measures}\label{Sec:MM}
From here on, we will adopt the following notations:
\begin{align*}
n_2(s)&:=2m((x+x^{-1})(y+y^{-1})(z+z^{-1})+s^{1/2}),\\
n_3(s)&:=m\left((x+x^{-1})^2(y+y^{-1})^2(1+z)^3z^{-2}-s\right),\\
n_4(s)&:=4m\left(x^4+y^4+z^4+1+s^{1/4}xyz\right),\\
E_s&: y^2=(x-1)\left(x^2-\frac{s}{s-64}\right),\\
F_s&: x^3+y^3+1-rxy=0,  \qquad r=\sqrt[3]{\frac{s+\sqrt{s(s-108)}}{2}},\\
G_s&: y^2=(x-1)(x-r')(x+r'),  \qquad r'=\sqrt{\frac{1+\sqrt{1-\frac{256}{s}}}{2}},\\
\chi_{D}(n)&=\left(\frac{D}{n}\right),\qquad d_k:=L'(\chi_{-k},-1),\qquad M_N := L'(g_N,0), 
\end{align*}
where $g_N$ is a normalized newform with rational Fourier coefficients in $S_3(\Gamma_0(N),\chi_{-N})$. 
The main goal of this section is to give a proof of Theorem~\ref{T:Main}. The key idea of the proof is to use the fact that when $s$ is properly parametrized, the Mahler measures $n_j(s), j=2,3,4$ can be expressed as Eisenstein-Kronecker series, which turns out to equal the series obtained from the right-hand sides of \eqref{E:Main},\eqref{E:Main2},\eqref{E:Main31}, and \eqref{E:Main32}. Let us first state a modified version of \cite[Prop.~2.1]{SamartII} below. 
\begin{proposition}\label{P:Samart}
For $\tau\in\mathcal{H},$ let $q=q(\tau):=e^{2\pi i\tau}$ and denote
\begin{align*}
s_2(q)&=-\frac{\Delta\left(\frac{2\tau+1}{2}\right)}{\Delta(2\tau+1)},\\
s_3(q)&=\left(27\left(\frac{\eta(3\tau)}{\eta(\tau)}\right)^6+\left(\frac{\eta(\tau)}{\eta(3\tau)}\right)^6\right)^2,\\
s_4(q)&=\frac{\Delta(2\tau)}{\Delta(\tau)}\left(16\left(\frac{\eta(\tau)\eta(4\tau)^2}{\eta(2\tau)^3}\right)^4+\left(\frac{\eta(2\tau)^3}{\eta(\tau)\eta(4\tau)^2}\right)^4\right)^4,
\end{align*}
where $\Delta(\tau)=\eta^{24}(\tau)$ and $\eta(\tau)$ is the Dedekind eta function. 
\begin{enumerate}
\item[(i)]
If $\tau=y_1i$ or $\tau=1/2+y_2i$, where $y_1\in[1/2,\infty)$ and $y_2\in(0,\infty)$, then  
\begin{multline*}
n_2(s_2(q))=\frac{2\Im(\tau)}{\pi^3}\sideset{}{'}\sum_{m,n\in \mathbb{Z}}\biggl(16\left(\frac{4(4n\Re(\tau)+m)^2}{[(4n\tau+m)(4n\bar{\tau}+m)]^3}-\frac{1}{[(4n\tau+m)(4n\bar{\tau}+m)]^2}\right)\\
-\left(\frac{4(n\Re(\tau)+m)^2}{[(n\tau+m)(n\bar{\tau}+m)]^3}-\frac{1}{[(n\tau+m)(n\bar{\tau}+m)]^2}\right)\biggr).
\end{multline*}
\item[(ii)]
If $\tau=yi$, where $y\in[1/\sqrt{3},\infty)$ then 
\begin{multline*}
n_3(s_3(q))=\frac{15\Im(\tau)}{4\pi^3}\sideset{}{'}\sum_{m,n\in \mathbb{Z}}\biggl(9\left(\frac{4(3n\Re(\tau)+m)^2}{[(3n\tau+m)(3n\bar{\tau}+m)]^3}-\frac{1}{[(3n\tau+m)(3n\bar{\tau}+m)]^2}\right)\\
-\left(\frac{4(n\Re(\tau)+m)^2}{[(n\tau+m)(n\bar{\tau}+m)]^3}-\frac{1}{[(n\tau+m)(n\bar{\tau}+m)]^2}\right)\biggr).
\end{multline*}
\item[(iii)]
If $\tau=y_1i$ or $\tau=1/2+y_2i$, where $y_1\in[1/\sqrt{2},\infty)$ and $y_2\in(1/2,\infty)$, then  
\begin{multline*}
n_4(s_4(q))=\frac{10\Im(\tau)}{\pi^3}\sideset{}{'}\sum_{m,n\in \mathbb{Z}}\biggl(4\left(\frac{4(2n\Re(\tau)+m)^2}{[(2n\tau+m)(2n\bar{\tau}+m)]^3}-\frac{1}{[(2n\tau+m)(2n\bar{\tau}+m)]^2}\right)\\
-\left(\frac{4(n\Re(\tau)+m)^2}{[(n\tau+m)(n\bar{\tau}+m)]^3}-\frac{1}{[(n\tau+m)(n\bar{\tau}+m)]^2}\right)\biggr).
\end{multline*}
\end{enumerate} 
\end{proposition}
\begin{proof}
The main part of the proof follows directly from that of \cite[Prop.~2.1]{SamartII}. Then the assumption can be slightly modified by using the fact that $s_2(q)=j_4^*(\tau)+24, s_3(q)=j_3^*(\tau)+42,$ and $s_4(q)=j_2^*(\tau)+104$ where
\begin{align*}
j_2^*(\tau)&=\left(\frac{\eta(\tau)}{\eta(2\tau)}\right)^{24}+24+2^{12}\left(\frac{\eta(2\tau)}{\eta(\tau)}\right)^{24},\\
j_3^*(\tau)&=\left(\frac{\eta(\tau)}{\eta(3\tau)}\right)^{12}+12+3^{6}\left(\frac{\eta(3\tau)}{\eta(\tau)}\right)^{12},\\
j_4^*(\tau)&=\left(\frac{\eta(\tau)}{\eta(4\tau)}\right)^8+8+4^4\left(\frac{\eta(4\tau)}{\eta(\tau)}\right)^8.
\end{align*}
It is well established that for each $N=2,3,4,$ $j_N^*(\tau)$ is a Hauptmodul associated to the genus zero subgroup $\Gamma_0(N)^*$ of $GL_2(\mathbb{R})$ generated by $\Gamma_0(N)$ and the Atkin-Lehner involutions $W_p,p|N$.
(see, for example, \cite[\S8]{ZagierIII}). Since the values of $\tau$ in the assumptions above lie in a corresponding fundamental domain for $\Gamma_0(N)^*$, the proposition follows by similar arguments in \cite[\S 14]{RV}.
\end{proof}

Before proving the main theorem, let us briefly discuss the significance of the families of elliptic curves $E_s, F_s,$ and $G_s$, which appear in the theorem. Note first that $(x+x^{-1})(y+y^{-1})(z+z^{-1})+s^{1/2}=0$ defines a family of $K3$ surfaces, say $X_s$, which has generic Picard number $19$, and $E_s$ is a family of elliptic curves giving rise to a Shioda-Inose structure of $X_s$. For many algebraic values of $s$, $E_s$ is a CM elliptic curve, so $X_s$ is singular; i.e. it has Picard number $20$, and $n_2(s)$ is (conjecturally) equal to a rational linear combination of $L$-values of types $M_N$ and $d_k$. Some of these formulas can be proved using Proposition~\ref{P:Samart}. However, for values of $s$ such that $E_s$ is non-CM, no Mahler measure $n_2(s)$ is provably known to be related to special values of $L$-functions. For more details, see \cite[\S 4]{Samart}. This is a part of the initial motivation to find a general formula for $n_2(s)$, which potentially involves $E_s$. \\
\indent Similarly, the zero loci of the other two families of polynomials also define one-parameter families of $K3$ hypersurfaces, namely
\begin{align*}
Y_s&: (x+x^{-1})^2(y+y^{-1})^2(1+z)^3z^{-2}-s=0,\\
Z_s&: x^4+y^4+z^4+1+s^{1/4}xyz=0.
\end{align*}
Furthermore, it can be shown that the family of elliptic curves $G_s$ gives rise to a Shioda-Inose structure associated to the quartic surfaces $Z_s$ by the following arguments:
Let $\mu=1/s$. Consider the integral 
\begin{equation*}
w_0(\mu):=\frac{1}{(2\pi i)^3}\int_{\mathbb{T}^3}\frac{1}{1-\mu^{1/4}\left(\frac{x^4+y^4+z^4+1}{xyz}\right)}\frac{dx}{x}\frac{dy}{y}\frac{dz}{z},
\end{equation*}
which can be realized as a formal period $\int_\gamma \omega_s$, where $\omega_s$ is a holomorphic $2$-form and $\gamma$ is a $2$-cycle on 
$Z_s$. Then using the Taylor series expansion and combinatorial arguments, one can find easily that for $|\mu|$ sufficiently small 
\begin{equation*}
w_0(\mu)=\pFq{3}{2}{\frac{1}{4},\frac{1}{2},\frac{3}{4}}{1,1}{256\mu},
\end{equation*}
where  
\begin{equation*}
\pFq{p}{q}{a_1,a_2,\ldots,a_p}{b_1,b_2,\ldots,b_q}{x}=\sum_{n=0}^{\infty}\frac{(a_1)_n\cdots(a_p)_n}{(b_1)_n\cdots(b_q)_n}\frac{x^n}{n!}
\end{equation*}
and $(c)_n=\Gamma(c+n)/\Gamma(c)$. Therefore, $w_0(\mu)$ satisfies the differential equation
\begin{equation}\label{E:PFquartic}
\mu^2(256\mu-1)\frac{d^3w}{d\mu^3}+\mu(1152\mu-3)\frac{d^2w}{d\mu^2}+(816\mu-1)\frac{dw}{d\mu}+24w=0.
\end{equation}
In other words, \eqref{E:PFquartic} is the Picard-Fuchs equation of the quartic surfaces. By direct calculation, one sees that \eqref{E:PFquartic} is the symmetric square of the second-order differential equation
\begin{equation}\label{E:PFClausen}
\mu(256\mu-1)\frac{d^2w}{d\mu^2}+(384\mu-1)\frac{dw}{d\mu}+12w=0,
\end{equation}
whose non-holomorphic solution around $\mu=0$ is 
$$\pFq{2}{1}{\frac{1}{4},\frac{3}{4}}{1}{\frac{1+\sqrt{1-256\mu}}{2}}.$$
It was obtained in the proof of \cite[Cor.~2.2]{McCarthy} that if $E_\lambda$ denotes the \textit{Clausen form} elliptic curves $$y^2=(x-1)(x^2+\lambda),\qquad \lambda\notin\{0,-1\},$$ then the real period $\Omega(E_\lambda)$ of $E_\lambda$ is 
\begin{equation*}
\Omega(E_\lambda)=\pi \pFq{2}{1}{\frac{1}{4},\frac{3}{4}}{1}{-\lambda}.
\end{equation*}
Hence \eqref{E:PFClausen} is the Picard-Fuchs equation of the family $G_s$, and this family of elliptic curves is associated to the $K3$ surfaces $Z_s$ by a Shioda-Inose structure (see \cite[\S 5]{LongII}). It is also worth mentioning that if we set $s=-2^{10}u^4/(u^4-1)^2$, then the $j$-function of the family $G_s$ is given by
\begin{equation*}
j(G_{s(u)})=\frac{64(u^2+3)^3(3u^2+1)^3}{(u^2-1)^4(u^2+1)^2},
\end{equation*}
which coincides with Long's result \cite[5.15]{LongI}.\\
\indent Finally, consider the family $Y_s.$ Again, its formal period is 
\begin{align*}
v_0(\mu):&=\frac{1}{(2\pi i)^3}\int_{\mathbb{T}^3}\frac{1}{1-\mu(x+x^{-1})^2(y+y^{-1})^2(1+z)^3z^{-2}}\frac{dx}{x}\frac{dy}{y}\frac{dz}{z}\\
&=\pFq{3}{2}{\frac{1}{3},\frac{1}{2},\frac{2}{3}}{1,1}{108\mu},
\end{align*}
which is a solution of
\begin{equation*}
\mu^2(108\mu-1)\frac{d^3v}{d\mu^3}+3\mu(162\mu-1)\frac{d^2v}{d\mu^2}+(348\mu-1)\frac{dv}{d\mu}+12v=0.
\end{equation*}
Recall from \cite[\S 14]{RV} that the Picard-Fuchs equation of the reparametrized \textit{Hesse form} elliptic curves $F_s$ is
\begin{equation*}
\mu(108\mu-1)\frac{d^2v}{d\mu^2}+(162\mu-1)\frac{dv}{d\mu}+6v=0.
\end{equation*}
By the same arguments above, it can be shown that the family $F_s$ of elliptic curves gives rise to a Shioda-Inose structure of $Y_s$.\\

The results below are essentially required in the proof of our main theorem.
\begin{lemma}\label{L:period}
For $t=2,3,$ and $4$, we denote 
$$\F_t(z):=\pFq{2}{1}{\frac{1}{t},\frac{t-1}{t}}{1}{z}.$$
\begin{enumerate}
\item [(i)]
Let $a,b,c\in\mathbb{R}$ be such that $a>b>c$ and let $E$ be the elliptic curve $y^2=(x-a)(x-b)(x-c).$ Then $E$ is isomorphic to $\mathbb{C}/(\mathbb{Z}+\mathbb{Z}\tau),$ where 
$$\tau=\frac{\F_2\left(\frac{a-b}{a-c}\right)}{\F_2\left(\frac{b-c}{a-c}\right)}i.$$
\item [(ii)]
Let $k\in (3,\infty)$ and let $E$ be the elliptic curve $x^3+y^3+1-kxy=0.$ Then $E$ is isomorphic to $\mathbb{C}/(\mathbb{Z}+\mathbb{Z}\tau),$ where 
$$\tau=\sqrt{3}\frac{\F_3\left(\frac{27}{k^3}\right)}{\F_3\left(1-\frac{27}{k^3}\right)}i.$$
\end{enumerate}
\end{lemma}
\begin{proof}
First, consider the elliptic curve $E: y^2=(x-a)(x-b)(x-c),$ where $a>b>c$. Let $\omega_1$ and $\omega_2$ be the real and complex periods of $E$, respectively, and let $\tau_1=\omega_2/\omega_1.$ Then we can find $\omega_1$ and $\omega_2$ using the following formulas (see, for example, \cite[Algorithm~7.4.7]{Cohen}):
\begin{equation*}
\omega_1=\frac{\pi}{\AGM(\sqrt{a-c},\sqrt{a-b})},\qquad \omega_2=\frac{i \pi}{\AGM(\sqrt{a-c},\sqrt{b-c})},
\end{equation*}
where $\AGM(\alpha,\beta)$ denotes the arithmetic-geometric mean of $\alpha$ and $\beta$, defined as follows: Let $(a_n)_{n=1}^{\infty}$ and $(b_n)_{n=1}^{\infty}$ be sequences given by $a_1=\alpha,b_1=\beta,$ and $a_{n+1}=(a_n+b_n)/2,b_{n+1}=\sqrt{a_n b_n}$ for $n\geq 1$. Then these two sequences converge to the same number, and  we call this number $\AGM(\alpha,\beta)$. 

By \cite[(3.5)]{McCarthy}, one has that the $\AGM$ can be represented by a $_2F_1$-hypergeometric series, namely,
\begin{equation}\label{E:Dermot}
\AGM(\alpha,\beta)=\frac{\alpha}{\F_2\left(1-\left(\frac{\beta}{\alpha}\right)^2\right)}.
\end{equation} 
Therefore, we have immediately that
\begin{equation*}
\tau_1=\frac{\AGM(\sqrt{a-c},\sqrt{a-b})}{\AGM(\sqrt{a-c},\sqrt{b-c})}i=\frac{\F_2\left(\frac{a-b}{a-c}\right)}{\F_2\left(\frac{b-c}{a-c}\right)}i,
\end{equation*}
and (i) is proved. 

Now for a given $k\in (3,\infty)$, let $\tau_2=\left(\sqrt{3}\F_3\left(\frac{27}{k^3}\right)/\F_3\left(1-\frac{27}{k^3}\right)\right)i,$ and let $E$ be the elliptic curve defined by $x^3+y^3+1-kxy=0.$ To establish (ii), we will show that $j(E)=j(\tau_2),$ where the latter $j$ is the usual $j$-invariant. Let us first introduce a generalized Weber function 
\begin{equation*}
\mathfrak{g}_3(\tau)=\sqrt{3}\frac{\eta(3\tau)}{\eta(\tau)}.
\end{equation*}
It is a classical result due to Weber that for any $\tau\in\mathcal{H},$ $\mathfrak{g}_3^{12}(\tau)$ is a zero of the polynomial $x^4+36x^3+270x^2+(756-j(\tau))x+3^6$ (see, for example, \cite[Thm.~5]{Uzunkol}). Consequently, we can write $j(\tau)$ as a rational function of $\mathfrak{g}_3(\tau)$, namely, 
\begin{equation}\label{E:j}
j(\tau)=\frac{\left(\mathfrak{g}_3^{12}(\tau)+3\right)^3\left(\mathfrak{g}_3^{12}(\tau)+27\right)}{\mathfrak{g}_3^{12}(\tau)}.
\end{equation}
Observe that we can also rewrite the function $s_3(q)$, defined in Proposition~\ref{P:Samart}, as
\begin{equation*}
s_3(q)=\left(\mathfrak{g}_3^{6}(\tau)+27\mathfrak{g}_3^{-6}(\tau)\right)^2.
\end{equation*}
Recall from Ramanujan's theory of elliptic functions of signature 3 that if $q_t(\alpha)$ is the \textit{elliptic nome}
\begin{equation}\label{E:ellnome}
q_t(\alpha)=\exp\left(-\frac{\pi}{\sin\left(\frac{\pi}{t}\right)}\frac{\F_t(1-\alpha)}{\F_t(\alpha)}\right),
\end{equation} 
then $s_3(q_3(\alpha))=\frac{27}{\alpha(1-\alpha)}$ for any $\alpha$ which makes both $\F_3(1-\alpha)$ and $\F_3(\alpha)$ convergent. Hence it follows that 
\begin{align*}
\left(\mathfrak{g}_3^{6}\left(-\frac{1}{\tau_2}\right)+27\mathfrak{g}_3^{-6}\left(-\frac{1}{\tau_2}\right)\right)^2&=s_3\left(q\left(-\frac{1}{\tau_2}\right)\right)\\
&=s_3\left(q_3\left(\frac{27}{k^3}\right)\right)\\
&=\frac{k^6}{k^3-27},
\end{align*}
which implies that $\mathfrak{g}_3^{12}\left(-\frac{1}{\tau_2}\right)$ can possibly be either $k^3-27$ or $729/(k^3-27)$. However, as a function of $k$, $\mathfrak{g}_3^{12}\left(-\frac{1}{\tau_2}\right)$ is decreasing on $(3,\infty)$, so it must equal $729/(k^3-27)$ on this interval. Therefore, we have by \eqref{E:j} that 
\begin{equation*}
j(\tau_2)=j\left(-\frac{1}{\tau_2}\right)=\left(\frac{k(k^3+216)}{k^3-27}\right)^3.
\end{equation*} 
On the other hand, it can be found using standard computer algebra systems such as \texttt{Maple} that if $k^3-27\neq 0$, then $j(E)$ coincides with $j(\tau_2)$ obtained above. 
\end{proof}

\begin{proof}[Proof of Theorem~\ref{T:Main}]
(i) Let $s\in\mathbb{R}\backslash [0,64],r=\sqrt{\frac{s}{s-64}},P=(-r,0),$ and $Q=(r,0).$ Then the equation representing $E_s$ can be rewritten as
\begin{equation}
E_s : y^2=(x-1)(x-r)(x+r).
\end{equation}
We will divide the proof of into two cases, whose arguments are somewhat parallel. 
\begin{case}\label{C:1}
$s>64$
\end{case}
In this case, we have $r>1.$ Then using Lemma~\ref{L:period}(i) it follows that $E_s\cong \mathbb{C}/(\mathbb{Z}+\mathbb{Z}\tau),$ where
\begin{equation*}
\tau=\frac{\F_2\left(\frac{r-1}{2r}\right)}{\F_2\left(\frac{r+1}{2r}\right)}i.
\end{equation*}
By a result from Ramanujan's theory of elliptic functions of signature 2, we have that $s_2(q_2(\alpha))=\frac{16}{\alpha(1-\alpha)},$ where $q_2(\alpha)$ is as defined in \eqref{E:ellnome}. As a consequence, we can easily deduce that
\begin{equation*}
s_2\left(q\left(-\frac{1}{2\tau}\right)\right)=s_2\left(\exp\left(-\pi\frac{\F_2\left(\frac{r+1}{2r}\right)}{\F_2\left(\frac{r-1}{2r}\right)}\right)\right)=s_2\left(q_2\left(\frac{r-1}{2r}\right)\right)=\frac{64r^2}{r^2-1}=s.
\end{equation*}
Since $\tau$ is purely imaginary and $0<\Im(\tau)<1$, it follows that $-\frac{1}{2\tau}$ is also purely imaginary, and 
\begin{equation*}
\left|-\frac{1}{2\tau}\right|=\Im\left(-\frac{1}{2\tau}\right)=\frac{1}{2\Im(\tau)}>\frac{1}{2}.
\end{equation*}
Let $y=\Im(\tau).$ Then, by applying Proposition~\ref{P:Samart}(i), one sees that
\begin{align*}
n_2\left(s_2\left(q\left(-\frac{1}{2\tau}\right)\right)\right)&=\frac{1}{y\pi^3}\sideset{}{'}\sum_{m,n\in \mathbb{Z}}\biggl(16\left(\frac{4m^2}{\left(\left(4n^2/y^2\right)+m^2\right)^3}-\frac{1}{\left(\left(4n^2/y^2\right)+m^2\right)^2}\right)\\
&\qquad\qquad\qquad -\left(\frac{4m^2}{\left(\left(n^2/4y^2\right)+m^2\right)^3}-\frac{1}{\left(\left(n^2/4y^2\right)+m^2\right)^2}\right)\biggr)\\
&=\frac{16y^3}{\pi^3}\sideset{}{'}\sum_{m,n\in \mathbb{Z}}\biggl(\left(\frac{4y^2m^2}{\left(4n^2+y^2m^2\right)^3}-\frac{1}{\left(4n^2+y^2m^2\right)^2}\right)\\
&\qquad\qquad\qquad -\left(\frac{16y^2m^2}{\left(n^2+4y^2m^2\right)^3}-\frac{1}{\left(n^2+4y^2m^2\right)^2}\right)\biggr)\\
&=\frac{16y^3}{\pi^3}\left[\vphantom{-\sideset{}{'}\sum_{\substack{m \text{ even}\\ n \in \mathbb{Z}}}\left(\frac{4y^2m^2}{\left(n^2+y^2m^2\right)^3}-\frac{1}{\left(n^2+y^2m^2\right)^2}\right)}\sideset{}{'}\sum_{\substack{m\in \mathbb{Z}\\n \text{ even}}}\left(\frac{4y^2m^2}{\left(n^2+y^2m^2\right)^3}-\frac{1}{\left(n^2+y^2m^2\right)^2}\right)\right.\\
&\qquad\qquad \left.-\sideset{}{'}\sum_{\substack{m \text{ even}\\n \in \mathbb{Z}}}\left(\frac{4y^2m^2}{\left(n^2+y^2m^2\right)^3}-\frac{1}{\left(n^2+y^2m^2\right)^2}\right)\right]\\
&=\frac{16y^3}{\pi^3}\left[\sum_{\substack{m \text{ even}\\n \text{ odd}}}\frac{n^2-3y^2m^2}{\left(n^2+y^2m^2\right)^3}-\sum_{\substack{m \text{ odd}\\n \text{ even}}}\frac{n^2-3y^2m^2}{\left(n^2+y^2m^2\right)^3}\right].
\end{align*}
On the other hand, recall that if $\tilde{E}$ is the elliptic curve
\begin{equation*}
Y^2=4X^3-g_2X-g_3=4(X-e_1)(X-e_2)(X-e_3),
\end{equation*}
where $e_1,e_2,e_3\in\mathbb{R}$ with $e_3<e_2<e_1$, and $\tilde{\omega_1}$ and $\tilde{\omega_2}$ are the real and complex periods of $\tilde{E}$, then
\begin{align*}
\left(\wp_\Lambda(\tilde{\omega_1}/2),\wp_\Lambda'(\tilde{\omega_1}/2)\right)&=(e_1,0),\\
\left(\wp_\Lambda((\tilde{\omega_1}+\tilde{\omega_2})/2),\wp_\Lambda'((\tilde{\omega_1}+\tilde{\omega_2})/2)\right)&=(e_2,0),\\
\left(\wp_\Lambda(\tilde{\omega_2}/2),\wp_\Lambda'(\tilde{\omega_2}/2)\right)&=(e_3,0),
\end{align*}
where $\Lambda=\mathbb{Z}\tilde{\omega_1}+\mathbb{Z}\tilde{\omega_2}.$ It can be checked in a straightforward manner that the birational map
\begin{equation*}
(x,y)\mapsto (36x-12,432y)
\end{equation*}
gives an isomorphism between $E_s(\mathbb{R})$ and 
\begin{align*}
\tilde{E_s}(\mathbb{R}) : y^2&=4x^3-(5184r^2+1728)x-(13824-124416r^2)\\
&=4(x-24)(x-(36r-12))(x-(-36r-12)).
\end{align*}
Since $-r<1<r$, one finds immediately that the isomorphism $E_s\cong \mathbb{C}/(\mathbb{Z}+\mathbb{Z}\tau)$ sends $P$ to $\tau/2$ and $Q$ to $1/2$. Let $\xi=(P)-(Q)$ and $q=e^{2\pi i \tau}$. Then it follows from Corollary~\ref{Co:2} that
\begin{align*}
\frac{8}{3\pi^2}\left(6\mathcal{L}^{E_s}_{3,1}-\mathcal{L}^{E_s}_{3,2}\right)(\xi)
&=\frac{8\Im(\tau)^3}{\pi^3}\sideset{}{'}\sum_{m,n\in\mathbb{Z}}\biggl(\frac{\left((-1)^m-(-1)^n\right)\left(2(m\Re(\tau)+n)^2-2m^2\Im(\tau)^2\right)}{|m\tau+n|^6}\\
&\qquad\qquad\qquad -\frac{\left((-1)^m-(-1)^n\right)}{|m\tau+n|^4}\biggr)\\
&=\frac{16y^3}{\pi^3}\left[\vphantom{-\sum_{\substack{m \text{ odd}\\ n \text{ even}}}\left(\frac{2(n^2-y^2m^2)}{(n^2+y^2m^2)^3}-\frac{1}{(n^2+y^2m^2)^2}\right)}\sum_{\substack{m \text{ even}\\ n \text{ odd}}}\left(\frac{2(n^2-y^2m^2)}{(n^2+y^2m^2)^3}-\frac{1}{(n^2+y^2m^2)^2}\right)\right.\\
&\qquad\qquad\qquad \left.-\sum_{\substack{m \text{ odd}\\ n \text{ even}}}\left(\frac{2(n^2-y^2m^2)}{(n^2+y^2m^2)^3}-\frac{1}{(n^2+y^2m^2)^2}\right)\right]\\
&=\frac{16y^3}{\pi^3}\left[\sum_{\substack{m \text{ even}\\ n \text{ odd}}}\frac{n^2-3y^2m^2}{\left(n^2+y^2m^2\right)^3}-\sum_{\substack{m \text{ odd}\\ n \text{ even}}}\frac{n^2-3y^2m^2}{\left(n^2+y^2m^2\right)^3}\right],
\end{align*}
where we use the fact that $\Re\left((m\tau+n)^2\right)=(m\Re(\tau)+n)^2-m^2\Im(\tau)^2$ in the second equality. Therefore, the first case of the theorem is proved.

\begin{case}
$s<0$
\end{case}
In this case, we have $0<r<1.$ Then by Lemma~\ref{L:period}(i) one finds that $E_s\cong \mathbb{C}/(\mathbb{Z}+\mathbb{Z}\tau)$ where
\begin{equation*}
\tau=\frac{\F_2\left(\frac{1-r}{1+r}\right)}{\F_2\left(\frac{2r}{1+r}\right)}i.
\end{equation*}
Thus $\tau$ is again purely imaginary and $\Im(\tau)\in(0,\infty).$ Let $\alpha=\frac{1+r}{2r}.$ Then $\alpha >1,$ and
\begin{equation}\label{E:tau2}
\tau=\frac{\F_2\left(\frac{\alpha-1}{\alpha}\right)}{\F_2\left(\frac{1}{\alpha}\right)}i.
\end{equation}
Now we apply the hypergeometric transformations \cite[Thm.~2.2.5]{AAR} and \cite{Mathematica} to deduce the following identities:
\begin{align*}
\F_2\left(\frac{\alpha-1}{\alpha}\right)&=\alpha^{\frac{1}{2}}\F_2(1-\alpha)\\
\F_2\left(\frac{1}{\alpha}\right)&=\alpha^{\frac{1}{2}}\left(\F_2(\alpha)+\F_2(1-\alpha)i\right).
\end{align*}
Plugging the expressions above into \eqref{E:tau2}, one has immediately that
\begin{equation*}
\frac{\tau-1}{2\tau}=\frac{\F_2(\alpha)}{2\F_2(1-\alpha)}i.
\end{equation*}
By the same argument in Case~\ref{C:1}, we then obtain 
\begin{equation*}
s_2\left(q\left(\frac{\tau-1}{2\tau}\right)\right)=s_2\left(\exp\left(-\pi\frac{\F_2(\alpha)}{\F_2(1-\alpha)}\right)\right)\\
=\frac{16}{\alpha(1-\alpha)}\\
=\frac{64r^2}{r^2-1}=s.
\end{equation*}
Again, we let $y=\Im(\tau).$ It is easily seen that 
\begin{equation*}
\frac{\tau-1}{2\tau}=\frac{1}{2}+\frac{1}{2y}i.
\end{equation*}
Hence it follows by Proposition~\ref{P:Samart}(i) that
\begin{align*}
n_2(s)&=n_2\left(s_2\left(q\left(\frac{\tau-1}{2\tau}\right)\right)\right)\\
&=\frac{1}{y\pi^3}\sideset{}{'}\sum_{m,n\in \mathbb{Z}}\biggl(16\left(\frac{4(2n+m)^2}{\left(4n^2\left(1+1/y^2\right)+4mn+m^2\right)^3}-\frac{1}{\left(4n^2\left(1+1/y^2\right)+4mn+m^2\right)^2}\right)\\
&\qquad\qquad\qquad -16\left(\frac{4(n+2m)^2}{\left(n^2\left(1+1/y^2\right)+4mn+4m^2\right)^3}-\frac{1}{\left(n^2\left(1+1/y^2\right)+4mn+4m^2\right)^2}\right)\biggr)\\
&=\frac{16y^3}{\pi^3}\sideset{}{'}\sum_{m,n\in \mathbb{Z}}\biggl(\left(\frac{4y^2(2n+m)^2}{\left(y^2(2n+m)^2+(2n)^2\right)^3}-\frac{1}{\left(y^2(2n+m)^2+(2n)^2\right)^2}\right)\\
&\qquad\qquad\qquad -\left(\frac{4y^2(n+2m)^2}{\left(y^2(n+2m)^2+n^2\right)^3}-\frac{1}{\left(y^2(n+2m)^2+n^2\right)^2}\right)\biggr)\\
&=\frac{16y^3}{\pi^3}\left[\vphantom{-\sum_{\substack{m \text{ even}\\ n\text{ odd}}}\left(\frac{4y^2(m+n)^2}{\left(n^2+y^2(m+n)^2\right)^3}-\frac{1}{\left(n^2+y^2(m+n)^2\right)^2}\right)}\sum_{\substack{m \text{ odd}\\ n\text{ even}}}\left(\frac{4y^2(m+n)^2}{\left(n^2+y^2(m+n)^2\right)^3}-\frac{1}{\left(n^2+y^2(m+n)^2\right)^2}\right)\right.\\
&\qquad\qquad \left.-\sum_{\substack{m \text{ even}\\ n\text{ odd}}}\left(\frac{4y^2(m+n)^2}{\left(n^2+y^2(m+n)^2\right)^3}-\frac{1}{\left(n^2+y^2(m+n)^2\right)^2}\right)\right]\\
&=\frac{16y^3}{\pi^3}\left[\vphantom{-\sum_{\substack{m \text{ odd}\\ n\text{ odd}}}\left(\frac{4y^2m^2}{\left(n^2+y^2m^2\right)^3}-\frac{1}{\left(n^2+y^2m^2\right)^2}\right)}\sum_{\substack{m \text{ odd}\\ n\text{ even}}}\left(\frac{4y^2m^2}{\left(n^2+y^2m^2\right)^3}-\frac{1}{\left(n^2+y^2m^2\right)^2}\right)\right.\\
&\qquad\qquad \left.-\sum_{\substack{m \text{ odd}\\ n\text{ odd}}}\left(\frac{4y^2m^2}{\left(n^2+y^2m^2\right)^3}-\frac{1}{\left(n^2+y^2m^2\right)^2}\right)\right]\\
&=\frac{16y^3}{\pi^3}\left[\sum_{\substack{m \text{ odd}\\ n \text{ odd}}}\frac{n^2-3y^2m^2}{\left(n^2+y^2m^2\right)^3}-\sum_{\substack{m \text{ odd}\\ n \text{ even}}}\frac{n^2-3y^2m^2}{\left(n^2+y^2m^2\right)^3}\right].
\end{align*}
To evaluate $\mathcal{L}^{E_s}_{3,1}$ and $\mathcal{L}^{E_s}_{3,2}$ at $\xi:=(P)-(Q)$, we first use the fact that $-r<r<1$ and the argument in Case~\ref{C:1} to find that $P$ and $Q$ are mapped to $\tau/2$ and $(1+\tau)/2$, respectively, via the isomorphism $E_s\cong\mathbb{C}/(\mathbb{Z}+\mathbb{Z}\tau).$ This therefore yields
\begin{align*}
\frac{8}{3\pi^2}\left(6\mathcal{L}^{E_s}_{3,1}-\mathcal{L}^{E_s}_{3,2}\right)(\xi)
&=\frac{8\Im(\tau)^3}{\pi^3}\sideset{}{'}\sum_{m,n\in\mathbb{Z}}\biggl(\frac{\left((-1)^{n-m}-(-1)^n\right)\left(2(m\Re(\tau)+n)^2-2m^2\Im(\tau)^2\right)}{|m\tau+n|^6}\\
&\qquad\qquad\qquad -\frac{\left((-1)^{n-m}-(-1)^n\right)}{|m\tau+n|^4}\biggr)\\
&=\frac{8y^3}{\pi^3}\sideset{}{'}\sum_{m,n\in\mathbb{Z}}(-1)^n\left((-1)^m-1\right)\biggl(\frac{\left(2n^2-2y^2m^2\right)}{|m\tau+n|^6}-\frac{1}{|m\tau+n|^4}\biggr)\\
&=\frac{16y^3}{\pi^3}\left[\vphantom{-\sum_{\substack{m \text{ odd}\\ n \text{ even}}}\left(\frac{2(n^2-y^2m^2)}{(n^2+y^2m^2)^3}-\frac{1}{(n^2+y^2m^2)^2}\right)}\sum_{\substack{m \text{ odd}\\ n \text{ odd}}}\left(\frac{2(n^2-y^2m^2)}{(n^2+y^2m^2)^3}-\frac{1}{(n^2+y^2m^2)^2}\right)\right.\\
&\qquad\qquad\qquad \left.-\sum_{\substack{m \text{ odd}\\ n \text{ even}}}\left(\frac{2(n^2-y^2m^2)}{(n^2+y^2m^2)^3}-\frac{1}{(n^2+y^2m^2)^2}\right)
\right]\\
&=\frac{16y^3}{\pi^3}\left[\sum_{\substack{m \text{ odd}\\ n \text{ odd}}}\frac{n^2-3y^2m^2}{\left(n^2+y^2m^2\right)^3}-\sum_{\substack{m \text{ odd}\\ n \text{ even}}}\frac{n^2-3y^2m^2}{\left(n^2+y^2m^2\right)^3}\right]=n_2(s),
\end{align*}
as desired.\\

(ii) Let $s\in[108,\infty).$ By Lemma~\ref{L:period}(ii), one has $F_s\cong \mathbb{C}/(\mathbb{Z}+\mathbb{Z}\tau),$ where $$\tau=\sqrt{3}\frac{\F_3\left(\frac{27}{r^3}\right)}{\F_3\left(1-\frac{27}{r^3}\right)}i.$$ Since $s=r^6/(r^3-27),$ we see immediately from the proof of Lemma~\ref{L:period}(ii) that $s_3(q(-1/\tau))=s.$ Also, letting $y=\Im(\tau),$ we have $\Im(-1/\tau)=1/y\geq 1/\sqrt{3}.$ Hence we can apply Proposition~\ref{P:Samart}(ii) to show that
\begin{equation*}
n_3(s)=n_3\left(s_3\left(q\left(-\frac{1}{\tau}\right)\right)\right)=\frac{15y^3}{4\pi^3}\left(\sideset{}{'}\sum_{\substack{m,n \in \mathbb{Z}\\ 3\nmid n}}\frac{n^2-3y^2m^2}{(n^2+y^2m^2)^3}-8\sideset{}{'}\sum_{\substack{m,n \in \mathbb{Z}\\ 3|n}}\frac{n^2-3y^2m^2}{(n^2+y^2m^2)^3}\right).
\end{equation*}
Let $RHS$ denote the right-hand side of \eqref{E:Main2}. Then it can be seen using Corollary~\ref{Co:2} that if $\omega=e^{2\pi i/3},$ then
\begin{align*}
RHS=\frac{3y^3}{4\pi^3}&\sideset{}{'}\sum_{m,n\in\mathbb{Z}}\biggl(\frac{3\omega^n-27\omega^{-m}-54\omega^{n-m}-2}{|m\tau+n|^4}\\
&\qquad+4\Re\left(\frac{\left(9\omega^{-m}-6\omega^{n}+18\omega^{n-m}-1\right)(m\tau+n)^2}{|m\tau+n|^6}\right)\biggr).
\end{align*}
It then can be shown that $RHS=n_3(s)$ by using \eqref{E:omega} and properly rearranging the terms inside the summation above.\\

(iii) Similar to Theorem~\ref{T:Main}(i), the proof can be divided into two cases, depending on the value of $s$. Assume first that $s\geq 256.$ Then $r'\in [1/\sqrt{2},1).$ Hence by Lemma~\ref{L:period}(i) we have that
$G_s\cong \mathbb{C}/(\mathbb{Z}+\mathbb{Z}\tau),$ where
\begin{equation*}
\tau=\frac{\F_2\left(\frac{1-r'}{1+r'}\right)}{\F_2\left(\frac{2r'}{1+r'}\right)}i,
\end{equation*} 
so that $\Im(-1/2\tau) \geq 1/\sqrt{2}.$
Then using \cite[Thm.~9.1,Thm.~9.2]{BBG} and the fact that $s_4(q_4(\alpha))=\frac{64}{\alpha(1-\alpha)}$ we can deduce that \begin{align*}
s_4\left(q\left(-\frac{1}{2\tau}\right)\right)&=s_4\left(\exp\left(-\pi\frac{\F_2\left(\frac{2r'}{1+r'}\right)}{\F_2\left(\frac{1-r'}{1+r'}\right)}\right)\right)\\
&=s_4\left(\exp\left(-\sqrt{2}\pi\frac{\F_4(r'^2)}{\F_4(1-r'^2)}\right)\right)\\
&=s_4\left(q_4(1-r'^2)\right)=s.
\end{align*}
By elementary but tedious calculations analogous to those in the proof of Theorem~\ref{T:Main}(i), if $y=\Im(\tau),$ then
\begin{align*}
n_4\left(s_4\left(q\left(-\frac{1}{2\tau}\right)\right)\right)&=\frac{20y^3}{\pi^3}\left(4\sideset{}{'}\sum_{\substack{m,n \in \mathbb{Z}\\ m \text{ even}}}\frac{n^2-3y^2m^2}{(n^2+y^2m^2)^3}-\sideset{}{'}\sum_{m,n\in \mathbb{Z}}\frac{n^2-3y^2m^2}{(n^2+y^2m^2)^3}\right)\\
&=\frac{16}{9\pi^2}\big(15\mathcal{L}^{G_s}_{3,1}(2(P)-(2Q)+2(P+2Q))\\
&\qquad +\mathcal{L}^{G_s}_{3,2}(4(Q)-5(P)+2(2Q)+4(P+Q)-5(P+2Q))\big).
\end{align*}
Next, if $s<0$, then $r'>1,$ so the normalized period lattice of $G_s$ is generated by $1$ and $$\tau:=\frac{\F_2\left(\frac{r'-1}{2r'}\right)}{\F_2\left(\frac{r'+1}{2r'}\right)}i.$$
We employ the hypergeometric transformations \cite[Thm.~2.2.5]{AAR} and \cite{Mathematica} one more time to deduce that
\begin{equation*}
\frac{\tau-1}{2\tau}=\frac{\F_2\left(\frac{2r'}{r'+1}\right)}{\F_2\left(\frac{1-r'}{r'+1}\right)}i.
\end{equation*}
Hence, in this case, 
\begin{equation*}
s_4\left(q\left(\frac{\tau-1}{2\tau}\right)\right)=s,
\end{equation*}
by the same argument used for the case $s\geq 256.$ If $y=\Im(\tau),$ then $\Re((\tau-1)/2\tau)=1/2$ and $\Im((\tau-1)/2\tau)=1/2y >1/2.$ Finally, it remains to show that 
\begin{align*}
n_4\left(s_4\left(q\left(\frac{\tau-1}{2\tau}\right)\right)\right)&=\frac{20y^3}{\pi^3}\left(4\sideset{}{'}\sum_{\substack{m \text{ odd}\\ n \text{ even}}}\frac{n^2-3y^2m^2}{(n^2+y^2m^2)^3}-\frac{3}{4}\sideset{}{'}\sum_{m,n\in \mathbb{Z}}\frac{n^2-3y^2m^2}{(n^2+y^2m^2)^3}\right)\\ &=\frac{8}{9\pi^2}\big(30\mathcal{L}^{G_s}_{3,1}(2(2Q)-(P+2Q)+2(P))\\
&\qquad +\mathcal{L}^{G_s}_{3,2}(5(P+2Q)+8(Q)+8(P+Q)-11(2Q)-10(P))\big),
\end{align*}
which, again, requires only Proposition~\ref{P:Samart}(iii), Corollary~\ref{Co:2}, and some laborious work.
\end{proof}

It is also interesting to consider $n_j(s)$ for the real values of $s$ omitted from the results in Theorem~\ref{T:Main}. For instance, if $s\in\{0,64\}$, then $E_s$ is singular; i.e., it is no longer an elliptic curve, so $\mathcal{L}^{E_s}_{3,j},j=1,2$ are not defined. It is not difficult to see by direct calculation that $n_2(0)=0$. Also, the author has shown in \cite[Thm.~1.2]{SamartII} that $n_2(64)$ has a simple expression in terms of a modular $L$-value, namely $n_2(64)=8L'(g_{16},0)$, where $g_{16}(\tau)=\eta^6(4\tau).$ If $s\in (0,64)$; i.e., $(x-1)\left(x^2-\frac{s}{s-64}\right)$ has only one real root, the story turns out to be quite different. Indeed, we will see that Formula~\eqref{E:Main} is not true in this case by the following observation:
\begin{proposition}\label{P:Pete}
Let $s\in (0,64)$ and let $E_s,r,P,$ and $Q$ be as defined in Theorem~\ref{T:Main}(i). Then 
\begin{equation*}
\mathcal{L}^{E_s}_{3,j}(P)=\mathcal{L}^{E_s}_{3,j}(Q),\qquad j=1,2.
\end{equation*}
\end{proposition} 
Our proof of this proposition relies on the following facts:
\begin{lemma}
Let $\tau=1/2+yi,$ where $y\in\mathbb{R}.$ Then the following identities hold:
\begin{align}
&\sum_{\substack{m \text{ odd}\\ n\in\mathbb{Z}}}\frac{(-1)^n m^2}{|m\tau+n|^6}=0, \label{L:claim1}\\
&\sum_{\substack{m \text{ odd}\\ n\in\mathbb{Z}}}\frac{(-1)^n (m/2+n)^2}{|m\tau+n|^6}=0. \label{L:claim2}
\end{align}
\end{lemma}
\begin{proof}
Using the fact that $|z|=|-z|=|\bar{z}|$ for any $z\in\mathbb{C}$ and simple substitution, we find that
\begin{align*}
\sum_{\substack{m \text{ odd}\\ n\in\mathbb{Z}}}\frac{(-1)^n m^2}{|m\tau+n|^6}&=\sum_{m,n\in\mathbb{Z}}\frac{(-1)^n (2m+1)^2}{|(2m+1)(-1/2-yi)-n|^6}\\
&=\sum_{m,n\in\mathbb{Z}}\frac{(-1)^n (2m+1)^2}{|(2m+1)(1/2-yi)-(n+2m+1)|^6}\\
&=\sum_{m,n\in\mathbb{Z}}\frac{(-1)^{-n-1} (2m+1)^2}{|(2m+1)(1/2-yi)+n|^6}\\
&=-\sum_{\substack{m \text{ odd}\\ n\in\mathbb{Z}}}\frac{(-1)^n m^2}{|m\tau+n|^6}.
\end{align*}
Hence \eqref{L:claim1} follows. Then we apply \eqref{L:claim1} to show that
\begin{align*}
\sum_{\substack{m \text{ odd}\\ n\in\mathbb{Z}}}\frac{(-1)^n (m/2+n)^2}{|m\tau+n|^6}&=\sum_{\substack{m \text{ odd}\\ n\in\mathbb{Z}}}\left(\frac{(-1)^n m^2}{4|m\tau+n|^6}+\frac{(-1)^n n(m+n)}{|m\tau+n|^6}\right)\\
&=\sum_{\substack{m \text{ odd}\\ n \text{ even}}}\frac{nm}{|(m-n)\tau+n|^6}-\sum_{\substack{m \text{ even}\\ n \text{ odd}}}\frac{nm}{|(m-n)\tau+n|^6}.
\end{align*}
Next, we use similar tricks from the proof of \eqref{L:claim1} to argue that
\begin{align*}
\sum_{\substack{m \text{ even}\\ n \text{ odd}}}\frac{nm}{|(m-n)\tau+n|^6}&=\sum_{\substack{m \text{ even}\\ n \text{ odd}}}\frac{nm}{|(m-n)(-1/2-yi)-n|^6}\\
&=\sum_{\substack{m \text{ even}\\ n \text{ odd}}}\frac{nm}{|(m-n)(1/2-yi)-m|^6}\\
&=\sum_{\substack{m \text{ even}\\ n \text{ odd}}}\frac{nm}{|(n-m)(1/2-yi)+m|^6}\\
&=\sum_{\substack{m \text{ odd}\\ n \text{ even}}}\frac{nm}{|(m-n)\tau+n|^6}.
\end{align*}  
Thus we have \eqref{L:claim2}.
\end{proof}

\begin{proof}[Proof of Proposition~\ref{P:Pete}]
Applying the transformation $y\mapsto y/2$, we instead consider the family
\begin{equation*}
\tilde{E_s} : y^2=4(x-1)(x-r)(x+r)=4x^3-4x^2-4r^2x+4r^2.
\end{equation*}
Then we again find the period lattice of $\tilde{E_s}$ using \cite[Algorithm~7.4.7]{Cohen}. Indeed, one has immediately that 
\begin{equation*}
E_s\cong \tilde{E_s}\cong \mathbb{C}/(\mathbb{Z}+\mathbb{Z}\tau),
\end{equation*}
where 
\begin{equation*}
\tau=\frac{1}{2}+\frac{\AGM\left(2\sqrt[4]{1-r^2},\sqrt{2\left(\sqrt{1-r^2}+1\right)}\right)}{2\AGM\left(2\sqrt[4]{1-r^2},\sqrt{2\left(\sqrt{1-r^2}-1\right)}\right)}i.
\end{equation*}
Note that the arguments in the $\AGM$ above are all positive, since $r^2=\frac{s}{s-64}<0.$ By similar analysis in the proof of Theorem~\ref{T:Main}(i), it can be shown that $P$ and $Q$ are mapped to $(1+\tau)/2$ and $\tau/2$ via the isomorphism above. Now by \eqref{E:L31} one sees that
\begin{align*}
\mathcal{L}^{E_s}_{3,1}(P)&=\frac{2y^3}{3\pi}\sideset{}{'}\sum_{m,n\in\mathbb{Z}}\left(\frac{(-1)^{n-m}}{|m\tau+n|^4}-(-1)^{n-m}\frac{(m/2+n)^2-m^2y^2}{|m\tau+n|^6}\right),\\
\mathcal{L}^{E_s}_{3,1}(Q)&=\frac{2y^3}{3\pi}\sideset{}{'}\sum_{m,n\in\mathbb{Z}}\left(\frac{(-1)^{n}}{|m\tau+n|^4}-(-1)^{n}\frac{(m/2+n)^2-m^2y^2}{|m\tau+n|^6}\right).
\end{align*}
Therefore, by \eqref{L:claim1}
\begin{align*}
\mathcal{L}^{E_s}_{3,1}((Q)-(P))&=\frac{2y^3}{3\pi}\sum_{\substack{m \text{ odd}\\n\in\mathbb{Z}}}\left(\frac{(-1)^{n}}{|m\tau+n|^4}-(-1)^{n}\frac{(m/2+n)^2-m^2y^2}{|m\tau+n|^6}\right)\\
&=\frac{4y^3}{3\pi}\sum_{\substack{m \text{ odd}\\n\in\mathbb{Z}}}\frac{(-1)^n m^2 y^2}{|m\tau+n|^6}=0.
\end{align*}
Similarly, by \eqref{E:L32}, 
\begin{align*}
\mathcal{L}^{E_s}_{3,2}(P)&=\frac{y^3}{\pi}\sideset{}{'}\sum_{m,n\in\mathbb{Z}}\left(\frac{(-1)^{n-m}}{|m\tau+n|^4}+2(-1)^{n-m}\frac{(m/2+n)^2-m^2y^2}{|m\tau+n|^6}\right),\\
\mathcal{L}^{E_s}_{3,2}(Q)&=\frac{y^3}{\pi}\sideset{}{'}\sum_{m,n\in\mathbb{Z}}\left(\frac{(-1)^{n}}{|m\tau+n|^4}+2(-1)^{n}\frac{(m/2+n)^2-m^2y^2}{|m\tau+n|^6}\right);
\end{align*}
whence 
\begin{align*}
\mathcal{L}^{E_s}_{3,2}((Q)-(P))&=\frac{y^3}{\pi}\sum_{\substack{m \text{ odd}\\n\in\mathbb{Z}}}\left(\frac{(-1)^{n}}{|m\tau+n|^4}+2(-1)^{n}\frac{(m/2+n)^2-m^2y^2}{|m\tau+n|^6}\right)\\
&=\frac{y^3}{\pi}\sum_{\substack{m \text{ odd}\\n\in\mathbb{Z}}}(-1)^{n}\frac{3(m/2+n)^2-m^2y^2}{|m\tau+n|^6}=0,
\end{align*}
where the last equality follows from \eqref{L:claim1} and \eqref{L:claim2}.
\end{proof}

Although no general formula for $n_2(s)$ where $s\in (0,64)$ has been found, we were able to verify the following formulas numerically in \texttt{PARI} and \texttt{Maple}: 
\begin{align*}
n_2(1)&\stackrel{?}=-\frac{12}{7\pi^2}\mathcal{L}^{E_1}_{3,1}(4(R)-(O)),\\
n_2(16)&\stackrel{?}=-\frac{12}{\pi^2}\mathcal{L}^{E_{16}}_{3,1}(4(R)+(O)),
\end{align*}
where $R=(1,0)$ and $\stackrel{?}=$ means that they are equal to at least 75 decimal places. (Note that $E_1$ and $E_{16}$ are both CM elliptic curves.)

\section{Connection with special values of $L$-functions}\label{Sec:Lvalues}
In this section, we investigate relationships between evaluations of $\mathcal{L}^{E}_{3,j}, j=1,2$ and some special values of $L$-functions, the first evidence of which is the symmetric square $L$-function of $E$. It was verified numerically in \cite[\S 3]{MS} that for some non-CM elliptic curves $E\cong \mathbb{C}/(\mathbb{Z}+\mathbb{Z}\tau)$ there exist degree zero divisors $\xi_1$ and $\xi_2$ on $E$ such that 
\begin{equation}\label{E:MS}
\begin{vmatrix}
\Re\left(K_{1,3}(\tau;\xi_1)\right) & K_{2,2}(\tau;\xi_1) \\ 
\Re\left(K_{1,3}(\tau;\xi_2)\right) & K_{2,2}(\tau;\xi_2)
\end{vmatrix}\stackrel{?}\sim_{\mathbb{Q}^\times}\frac{\pi^6}{\Im(\tau)^4}L''(\Sym^2E,0).
\end{equation}
The above relation can be rephrased in terms of the determinant of $\mathcal{L}^{E}_{3,j}, j=1,2$ using the result below.
\begin{proposition}
Let $E$ be an elliptic curve over $\mathbb{C}$ and suppose that $E\cong \mathbb{C}/(\mathbb{Z}+\mathbb{Z}\tau)$. If $\xi_1$ and $\xi_2$ are divisors of degree zero on $E$, then 
\begin{equation*}
\begin{vmatrix}
\mathcal{L}^E_{3,1}(\xi_1) & \mathcal{L}^E_{3,2}(\xi_1) \\ 
\mathcal{L}^E_{3,1}(\xi_2) & \mathcal{L}^E_{3,2}(\xi_2)
\end{vmatrix}=-\frac{2\Im(\tau)^6}{\pi^2}\begin{vmatrix}
\Re\left(K_{1,3}(\tau;\xi_1)\right) & K_{2,2}(\tau;\xi_1) \\ 
\Re\left(K_{1,3}(\tau;\xi_2)\right) & K_{2,2}(\tau;\xi_2)
\end{vmatrix},
\end{equation*}
where $K_{a,b}(\tau;\xi)=\sum_{P\in E}n_P K_{a,b}(\tau;u_P)$ if $\xi=\sum_{P\in E}n_P (P)$ and $u_P$ is the image of $P$ in $\mathbb{C}/(\mathbb{Z}+\mathbb{Z}\tau)$.
\end{proposition}
\begin{proof}
For any degree zero divisor $\xi=\sum_{P\in E}n_P (P)$, it follows directly from \eqref{E:EK2} and \eqref{E:EK3} that
\begin{align*}
\Re\left(D_{1,3}(q;\xi)\right)&=2\mathcal{L}^E_{3,1}(\xi)-\frac{4}{3}\mathcal{L}^E_{3,2}(\xi),\\
D_{2,2}(q;\xi)&=-4\mathcal{L}^E_{3,1}(\xi)-\frac{4}{3}\mathcal{L}^E_{3,2}(\xi),
\end{align*}
where $q=e^{2\pi i\tau}$ and $D_{a,b}(q;\xi)=\sum_{P\in E}n_P D_{a,b}(q;e^{2\pi i u_P}).$ Then by Proposition~\ref{P:Zagier}(iii) and the two equations above one has that
\begin{align*}
\left(\frac{4\Im(\tau)^3}{\pi}\right)^2
\begin{vmatrix}
\Re\left(K_{1,3}(\tau;\xi_1)\right) & K_{2,2}(\tau;\xi_1) \\ 
\Re\left(K_{1,3}(\tau;\xi_2)\right) & K_{2,2}(\tau;\xi_2)
\end{vmatrix}&= 
\begin{vmatrix}
\Re\left(D_{1,3}(q;\xi_1)\right) & D_{2,2}(q;\xi_1) \\ 
\Re\left(D_{1,3}(q;\xi_2)\right) & D_{2,2}(q;\xi_2)
\end{vmatrix}\\
&=-8\begin{vmatrix}
\mathcal{L}^E_{3,1}(\xi_1) & \mathcal{L}^E_{3,2}(\xi_1) \\ 
\mathcal{L}^E_{3,1}(\xi_2) & \mathcal{L}^E_{3,2}(\xi_2)
\end{vmatrix}.
\end{align*}
\end{proof}
More generally, a conjecture relating $L(\Sym^n E,n+1)$ to determinants of Eisenstein-Kronecker series was formulated by Goncharov in \cite[\S 6]{Goncharov}. By the functional equation for $L(\Sym^2 E,s)$, the relation \eqref{E:MS} can be seen as a special case of this conjecture when $n=2$.

On the other hand, we observed from our computational experiments that if $E_s$ is CM, the functions $\mathcal{L}^{E_s}_{3,1}$ and $\mathcal{L}^{E_s}_{3,2}$ evaluated at some torsion divisors are individually related to lower degree $L$-values; i.e., the ones associated to Dirichlet characters and elliptic modular forms. These results are listed at the end of this section. In particular, some weaker results below are immediate consequences of Theorem~\ref{T:Main} and \cite[Thm.1.4]{SamartII}. Recall that $M_N$ and $d_k$ are as defined in Section~\ref{Sec:MM}, and superscripts will be attached to $M_N$ if there are more than one normalized newform of a given level $N$ with rational Fourier coefficients. 
\begin{proposition}
Let $E_s,F_s,$ and $G_s$ be the families of elliptic curves as defined in Section~\ref{Sec:MM}. 
\begin{enumerate}
\item[(i)]
Let $E$ denote $E_{256} : y^2=(x-1)\left(x^2-\frac{4}{3}\right)$, $P=\left(-\frac{2}{\sqrt{3}},0\right)$, and $Q=\left(\frac{2}{\sqrt{3}},0\right)$. Then we have
\begin{equation*}
\left(6\mathcal{L}_{3,1}^E-\mathcal{L}_{3,2}^E\right)((P)-(Q))=\frac{\pi^2}{2}(M_{48}+2d_4).
\end{equation*}
\item[(ii)]
Let $O,P$ and $Q$ are the points on $F_s$ corresponding to $1,1/3$ and $\tau/3$, respectively, via the isomorphism $F_s\cong \mathbb{C}/(\mathbb{Z}+\mathbb{Z}\tau), \tau\in\mathcal{H}$. \\
If $\mathcal{T}(s):=15\mathcal{L}^{F_s}_{3,1}((Q)-3(P)-6(P+Q))+\mathcal{L}^{F_s}_{3,2}(3(P)+6(P+Q)-7(Q)-2(O))$, then the following formulas are true:
\begin{align*}
\mathcal{T}(108)&=20\pi^2M_{12}, \qquad\qquad \mathcal{T}(216)=5\pi^2\left(M_{24}^{(2)}+d_3\right),\\
\mathcal{T}(1458)&=\frac{3\pi^2}{2}\left(9M_{12}+2d_4\right).
\end{align*}
\end{enumerate}
\item[(iii)] 
Let $P$ and $Q$ denote the points on $G_s$ corresponding to $\tau/2$ and $3/4$, respectively, via the isomorphism $G_s\cong \mathbb{C}/(\mathbb{Z}+\mathbb{Z}\tau)$. If we set
\begin{align*}
\mathcal{U}(s)&:=15\mathcal{L}^{G_s}_{3,1}(2(P)-(2Q)+2(P+2Q))\\ &\qquad+\mathcal{L}^{G_s}_{3,2}(4(Q)-5(P)+2(2Q)+4(P+Q)-5(P+2Q)),\\
\mathcal{V}(s)&:=30\mathcal{L}^{G_s}_{3,1}(2(2Q)-(P+2Q)+2(P))\\
&\qquad +\mathcal{L}^{G_s}_{3,2}(5(P+2Q)+8(Q)+8(P+Q)-11(2Q)-10(P)),
\end{align*}
then the following formulas are true:
\begin{align*}
\mathcal{U}(256)&=\frac{45\pi^2}{2}M_8, &\mathcal{U}(648)&=\frac{45\pi^2}{32}\left(4M_{16}+d_4\right),\\
\mathcal{U}(2304)&=\frac{15\pi^2}{4}\left(M_{24}^{(1)}+d_3\right), &\mathcal{U}(20736)&=\frac{9\pi^2}{20}\left(5M_{40}^{(1)}+2d_3\right),\\
\mathcal{U}(614656)&=\frac{15\pi^2}{2}\left(5M_8+d_3\right),& &\\
\mathcal{U}(3656+2600\sqrt{2})&=\frac{45\pi^2}{128}\left(4M_{32}+28M_8+4d_4+d_8\right),& &\\
\mathcal{V}(3656-2600\sqrt{2})&=\frac{45\pi^2}{64}\left(44M_{32}-28M_8+4d_4-d_8\right).& & 
\end{align*}
\end{proposition}
We will conclude this paper by listing some conjectural formulas for $\mathcal{L}_{3,j}^{E_s}$ evaluations at torsion points when $E_s$ is a CM elliptic curve over $\mathbb{Q}$. For each fixed $s$, we let $L_j=\frac{1}{\pi^2}\mathcal{L}_{3,j}^{E_s}$, and $O,P,$ and $Q$ denote the points on $E_s$ corresponding to $1,\tau/2,$ and $3/4,$ respectively, via the isomorphism $E_s\cong \mathbb{C}/(\mathbb{Z}+\mathbb{Z}\tau)$.  We firmly believe that some of these formulas could be verified rigorously using Corollary~\ref{Co:2} and double series expressions of $L$-series established in \cite{SamartII}.

\begin{table}[h]
\centering
    \begin{tabular}{ | c | c | c | c |}
    \hline
     & \tiny{$s=-512$} & \tiny{$s=-64$} & \tiny{$s=-8$} \\ \hline
    \tiny{$L_1(2Q)$} & $-\frac{(6M_{16}+5d_4)}{36}$ & $-\frac{(18M_{8}+d_8)}{36}$ & $-\frac{(6M_{16}+d_4)}{36}$ \\ \hline
    \tiny{$L_1(P)$} & $\frac{6M_{16}+6M_{64}-d_4+12d_8}{288}$ & $\frac{18M_{8}+6M_{32}+12d_4-d_8}{144}$ & $\frac{6M_{16}-d_4}{36}$ \\ \hline
    \tiny{$L_1(P+2Q)$} & $\frac{6M_{16}-6M_{64}-d_4-12d_8}{288}$ & $\frac{18M_{8}-6M_{32}-12d_4-d_8}{144}$ & $-\frac{d_4}{9}$ \\ \hline
    \tiny{$L_1(Q)$} & $-\frac{(6M_{16}+d_4)}{288}$ & $-\frac{(18M_{8}+6M_{32}-12d_4+d_8)}{576}$ & $-\frac{(6M_{16}+6M_{64}+d_4-12d_8)}{576}$ \\ \hline
    \tiny{$L_1(P+Q)$} & $-\frac{d_4}{72}$ & $-\frac{(18M_{8}-6M_{32}+12d_4+d_8)}{576}$ & $-\frac{(6M_{16}-6M_{64}+d_4+12d_8)}{576}$ \\ \hline
    \tiny{$L_1(O)$} & $\frac{6M_{16}+7d_4}{36}$ & $\frac{(6M_{8}+d_8)}{18}$ & $\frac{2d_4}{9}$ \\ \hline
    \end{tabular}
    \caption{Conjectured formulas of $L_1$}
\label{Tb:1}
\end{table}

\begin{table}[h]
\centering
    \begin{tabular}{ | c | c | c | c | c |}
    \hline
     & \tiny{$s=1$} & \tiny{$s=16$} & \tiny{$s=256$} &  \tiny{$s=4096$}\\ \hline
    \tiny{$L_1(2Q)$} & $-\frac{(48M_{7}-d_7)}{36}$ & $-\frac{(2M_{12}+d_3)}{12}$ & $\frac{2M_{12}-2M_{48}-d_3-8d_4}{72}$ & $\frac{48M_{7}-6M_{112}-96d_4-d_7}{504}$\\ \hline
    \tiny{$L_1(P)$} & $\frac{33M_{7}-2d_7}{36}$ & $\frac{M_{12}-d_3}{12}$ & $\frac{2M_{12}+2M_{48}-d_3+8d_4}{72}$ & $\frac{48M_{7}+6M_{112}+96d_4+d_7}{504}$\\ \hline
    \tiny{$L_1(P+2Q)$} & $\frac{33M_{7}-2d_7}{36}$ & $\frac{M_{12}-d_3}{12}$ & $-\frac{(2M_{12}+2d_3)}{9}$ & $-\frac{66M_{7}+4d_7}{63}$\\ \hline
    \tiny{$L_1(Q)$} & $-\frac{(48M_{7}+6M_{112}-96d_4-d_7)}{576}$ & $-\frac{(2M_{12}+2M_{48}+d_3-8d_4)}{192}$ & $(*)$ & $(*)$\\ \hline
    \tiny{$L_1(P+Q)$} & $-\frac{(48M_{7}-6M_{112}+96d_4-d_7)}{576}$ & $-\frac{(2M_{12}-2M_{48}+d_3+8d_4)}{192}$ & $(*)$ & $(*)$\\ \hline
    \tiny{$L_1(O)$} & $-\frac{(6M_{7}-d_7)}{9}$ & $\frac{d_3}{3}$ & $\frac{2M_{12}+3d_3}{9}$ & $\frac{72M_{7}+5d_7}{63}$\\ \hline
    \end{tabular}
    \caption{Conjectured formulas of $L_1$ (continued)}
\label{Tb:2}
\end{table}

$(*)$ When $s=256$ and $s=4096$, no individual conjectural formulas of $L_1(Q)$ and $L_1(P+Q)$ were detected in our numerical computations. However, we found that 
\begin{align*}
s&=256: &L_1((Q)+(P+Q))&\stackrel{?}=\frac{2M_{12}-2M_{48}-d_3-8d_4}{576},\\
s&=4096: &L_1((Q)+(P+Q))&\stackrel{?}=\frac{48M_{7}-6M_{112}-96d_4+d_7}{4032}.
\end{align*}

\begin{table}[h]
\centering
    \begin{tabular}{ | c | c |}
    \hline
    $s=-512$ & \tiny{$L_2((P)-(P+2Q))\stackrel{?}=\frac{M_{64}-d_8}{8},\qquad L_2(3(P+Q)-4(Q)+(O))\stackrel{?}=\frac{3M_{16}-d_4}{4}$}\\ \hline
    $s=-64$ & \tiny{$L_2((P)-(P+2Q))\stackrel{?}=\frac{M_{32}-d_4}{4},\qquad L_2((Q)-(P+Q))\stackrel{?}=\frac{M_{32}+d_4}{16}$}\\ \hline
    $s=-8$ & \tiny{$L_2((2Q)-(P))\stackrel{?}=M_{16},\qquad L_2((Q)+(P+Q)-2(O))\stackrel{?}=\frac{4M_{16}-43d_4}{64}$}\\ \hline
    $s=1$ & \tiny{$L_2((2Q)-(P))\stackrel{?}=\frac{54M_{7}+d_7}{8},\qquad L_2((Q)-(P+Q))\stackrel{?}=\frac{M_{112}+8d_4}{16}$}\\ \hline
    $s=16$ & \tiny{$L_2((2Q)-(P))\stackrel{?}=\frac{3M_{12}}{4},\qquad L_2((Q)-(P+Q))\stackrel{?}=\frac{M_{48}+2d_4}{16}$}\\ \hline
    $s=256$ & \tiny{$L_2((2Q)-(P))\stackrel{?}=\frac{M_{48}-2d_4}{6},\qquad L_2((Q)+(P+Q)-2(O))\stackrel{?}=\frac{508M_{12}+4M_{48}-385d_3-8d_4}{384}$}\\ \hline
    $s=4096$ & \tiny{$L_2((2Q)-(P))\stackrel{?}=\frac{M_{112}-8d_4}{14},\qquad L_2((Q)+(P+Q)-2(O))\stackrel{?}=\frac{6112M_{7}+4M_{112}-32d_4-213d_7}{896}$}\\ \hline
    \end{tabular}
    \caption{Conjectured formulas of $L_2$}
\label{Tb:3}
\end{table}

One of the remarkable features of these formulas is that they appear to support the conjectural relation \eqref{E:MS}. For instance, consider the CM elliptic curve $E:=E_{-8}$ of conductor $576$ with the corresponding $\tau=\sqrt{-1}.$ One can verify, at least numerically, that 
\begin{equation*}
L''(\Sym^2 E,0)\stackrel{?}=2d_{4} M_{16}.
\end{equation*}
Then, choosing $\xi_1=(Q)+(P+Q)-2(O)$ and $\xi_2=(2Q)-(P)$, we obtain the following identity directly from our formulas in Table~\ref{Tb:1} and Table~\ref{Tb:3}:
\begin{equation*}
\begin{vmatrix}
\mathcal{L}^E_{3,1}(\xi_1) & \mathcal{L}^E_{3,2}(\xi_1) \\ 
\mathcal{L}^E_{3,1}(\xi_2) & \mathcal{L}^E_{3,2}(\xi_2)
\end{vmatrix}\stackrel{?}=-\frac{43\pi^4}{64}d_{4} M_{16}\stackrel{?}=\frac{43\pi^4}{128}\Im(\tau)^2L''(\Sym^2 E,0).
\end{equation*}
To our knowledge, this particular example does not seem to appear in the literature, though it is exactly analogous to the numerical result due to Zagier which involves a non-CM elliptic curve of conductor $37$ \cite[\S 10]{ZG}. As a possible continuing research project, it would be interesting to find and prove this type of relations for some CM elliptic curves, which we will not pursue in this paper.\\

\textbf{Acknowledgements}
The author would like to thank Ken Ono for his encouragement during the author's visit at Emory University in the spring semester 2013. His comments and suggestions significantly lead to a motivation to continue this research project. The author also thanks Matthew Papanikolas for several helpful discussions and he is grateful to the anonymous referee for useful comments and advice which help improve the exposition of this paper.

\bibliographystyle{amsplain}

\end{document}